\documentclass{amsart}

\input prepictex
\input pictexwd    
\input postpictex

\usepackage{amsmath}
\usepackage{amssymb}
\usepackage{a4wide}
\usepackage{mathrsfs}
\usepackage{epic}
\usepackage{hyperref}
\usepackage{array,booktabs}
\usepackage{graphicx}
\usepackage{tikz}

\usepackage{mathtools}
\definecolor{green}{rgb}{0.2, 0.8, 0.2}

\usepackage{caption}
\newtheorem{Theorem}{Theorem}
\newtheorem{Proposition}{Proposition}
\newtheorem{Lemma}{Lemma}

\extrafloats{100}

\newtheorem{Remarks}{Remarks}
\newtheorem{Example}{Example}
\def\komma{\raise.5ex\hbox{,}}
\def\punt{\raise.5ex\hbox{.}}

\newcommand{\R}{\mathbb R}
\newcommand{\Z}{\mathbb Z}
\newcommand{\Q}{\mathbb Q}
\newcommand{\N}{\mathbb N}
\newcommand{\C}{\mathbb C}
\newcommand{\T}{T_{\alpha}}

\begin{document}

\title{Matching of orbits of certain $N$-expansions with a finite set of digits}

\author{Yufei Chen}
\address{School of Mathematical Sciences, East China Normal University,
Shanghai, 201100, P.R.\ China \& \newline
Delft University of Technology, EWI (DIAM), Mekelweg 4, 2628 CD Delft, the
Netherlands} \email{Y.Chen-18@tudelft.nl}

\author{Cor Kraaikamp}
\address{Delft University of Technology, EWI (DIAM), Mekelweg 4, 2628 CD Delft, the
Netherlands} \email{c.kraaikamp@tudelft.nl}

\date{\today}

\begin{abstract}
In this paper we consider a class of continued fraction expansions: the so-called \emph{$N$-expansions with a finite digit set}, where $N\geq 2$ is an integer. These \emph{$N$-expansions with a finite digit set} were introduced in~\cite{KL,L}, and further studied in~\cite{dJKN,S}. For $N$ fixed they are steered by a parameter $\alpha\in (0,\sqrt{N}-1]$. In~\cite{KL}, for $N=2$ an explicit interval $[A,B]$ was determined, such that for all $\alpha\in [A,B]$ the entropy $h(T_{\alpha})$ of the underlying Gauss-map $T_{\alpha}$ is equal. In this paper we show that for all $N\in\N$, $N\geq 2$, such plateaux exist. In order to show that the entropy is constant on such plateaux, we obtain the underlying planar natural extension of the maps $T_{\alpha}$, the $T_{\alpha}$-invariant measure, ergodicity, and we show that for any two $\alpha,\alpha'$ from the same plateau, the natural extensions are metrically isomorphic, and the isomorphism is given explicitly. The plateaux are found by a property called \emph{matching}.
\end{abstract}

\maketitle


\section{Introduction}\label{sec:Introduction}
It is well known that every real number $x$ can be written as a finite (in case $x\in\Q$) or infinite (regular) continued fraction of the form:
\begin{equation}\label{eq:RCF}
x=a_0+\frac{\displaystyle 1}{\displaystyle a_1+\frac{1}{a_2+\ddots
+\displaystyle \frac{1}{a_n+\ddots}}} = [a_0;a_1,a_2,\dots ,a_n,\dots ],
\end{equation}
where $a_0\in\Z$ such that $x-a_0\in [0,1)$, and $a_n\in\N$ for $n\geq 1$. Such a \emph{regular continued fraction expansion} (RCF) of $x$ is unique if and only if $x$ is irrational; in case $x\in\Q$ one has two expansions of the form~(\ref{eq:RCF}).\smallskip\

Apart from the regular continued fraction expansion algorithm there exist a bewildering number of other continued fraction expansion algorithms. In this paper we consider a recent algorithm, which was introduced by  Edward Burger and some of his students in 2008 in~\cite{Bu}.

Let $N\in\N_{\geq 2}$ be a fixed positive integer, and define the map $T_N:[0,1)\to [0,1)$ by:
\begin{equation}\label{def:NGaussMap}
T_N(x)=\frac{N}{x}- \left\lfloor \frac{N}{x} \right\rfloor ,\,\, x\neq 0;\quad T_N(0)=0.
\end{equation}
Setting $d_1=d_1(x)=\lfloor N/x\rfloor$, and $d_n=d_n(x)=d_1\left( T_N^{n-1}(x)\right)$, whenever $T_N^{n-1}(x)\neq
0$, we find:
\begin{equation}\label{eq:intr3easier}
x=\frac{\displaystyle N}{\displaystyle d_1+\frac{N}{d_2+\ddots
+\displaystyle \frac{N}{d_n+T_N^n(x)}}}.
\end{equation}
Taking finite truncations yield the convergents, which converge to $x$.

Burger \emph{et al}.\ studied these $N$-expansions, as they could show that for every quadratic irrational number $x$ there exists infinitely many $N\in\N$ for which the $N$-expansion of $x$ is ultimately periodic with period length 1. In 2011, Anselm and Weintraub further studied $N$-expansions in~\cite{AW}. They showed that every positive real number $x$ always has an $N$-expansion, and for $N\geq 2$ even infinitely many, and that rationals always have finite and infinite expansions. Furthermore, in case $N\geq 2$ every quadratic irrational has both periodic and non-periodic expansions. In their algorithm to find an $N$-expansion of a real number $x$ there is a \emph{best choice} for the partial quotient (i.e.\ digit), and if one always makes this best choice for the partial quotients one finds what they call the \emph{best expansion} of $x$. One can show that the $N$-expansions obtained via the Gauss-map $T_N$ from~(\ref{def:NGaussMap}) are always best expansions. Note that in~\cite{AW} $N$-expansions are not introduced or studied via maps such as defined in~(\ref{def:NGaussMap}). This was done in~\cite{DKW}, where many properties of $N$-expansions (such as ergodicity, the form of the invariant measure, entropy) were obtained in a very easy way.\smallskip\

In his MSc-thesis~\cite{L} from 2015, and in a subsequent paper with the second author~\cite{KL}, Niels Langeveld considered $N$-expansions on an interval \textbf{not} containing $0$. To be more precise: let $N \in \N_{\geq 2}$ and $\alpha \in \R$ such that $0 < \alpha \leq \sqrt{N}-1$, then we define $I_\alpha:=[\alpha,\alpha+1]$ and $I_\alpha^-:=[\alpha,\alpha+1)$ and investigate the continued fraction map $T_{\alpha}:I_{\alpha} \to I_{\alpha}^-$, defined as:
\begin{equation}\label{defNalphaGaussMap}
T_{\alpha}(x):= \frac{N}{x} - d(x),
\end{equation}
where $d: I_{\alpha} \to \N$ is defined by $d(x):=\left \lfloor \frac{N}x -\alpha\right \rfloor$.\smallskip\

Note that due to the fact that $\alpha >0$ there are only finitely many values of partial quotients $d$ possible. Furthermore, \emph{all} expansions are infinite. This new $N$-expansion (with a finite digit set) could be viewed as a small variation of the $N$-expansions with infinitely many digits, but actually the situation is suddenly dramatically different and more difficult as for certain values of $N$ and $\alpha$ ``\emph{gaps}'' in the interval $I_{\alpha}$ appear. We mention here an example from~\cite{dJKN}: take $N=51$, $\alpha = 6$. In this case there are only 2 digits (viz.\ 1 and 2), and setting for $n\geq 0$: $r_n=T_{\alpha}^n(\alpha +1)$, $\ell_n=T_{\alpha}(\alpha )$, and in general for a digit $i$:
$$
f_i=f_i(N) = \frac{\sqrt{4N+i^2}-i}{2},
$$
as the fixed point of $T_{\alpha}$ with digit $i$, then we immediately see two \emph{gaps} popping up:
$$
\begin{tikzpicture}[scale =10]
\draw[black,fill=black] (.2111,0) circle (.02ex);
\draw[black,fill=black] (.659,0) circle (.02ex);
  \draw [thick] (0,0) -- (0.159,0);
  \draw [dashed,red] (0.159,0) -- (0.2857,0);
  \draw [thick] (0.2857,0) -- (0.5,0);
   \draw [dashed,red] (0.5,0) -- (0.846,0);
\draw [thick] (0.846,0) -- (1,0);
  \draw (.375,-.015) -- (.375,.02);
 \draw (1,-.015) -- (1,.02);
  \draw (0,-.015) -- (0,.02);
   \draw (.159,-.005) -- (.159,.005);
    \draw (.2857,-.005) -- (.2857,.005);
 \draw (.846,-.005) -- (.846,.005);
 \draw (.5,-.005) -- (.5,.005);
\node at (.69,.03) {$\Delta_1$};
\node at (.19,.03) {$\Delta_2$};
\node at (0,-.03) {$\alpha$};
\node at (1,-.03) {$\alpha+1$};
\node at (.375,-.03) {$p_2$};
\node at (.659,-.03) {$f_1$};
\node at (.2111,-.03) {$f_2$};
\node at (.159,-.03) {$r_2$};
\node at (.286,-.03) {$r_1$};
\node at (.5,-.03) {$\ell_1$};
\node at (.846,-.03) {$\ell_2$};
\end{tikzpicture}
$$
In~\cite{dJKN} it has been investigated when these ``gaps'' appear. In~\cite{dJK} it will be shown that the number of gaps grows when $N\in\N$ increases. In spite of the gaps, in~\cite{dJKN} the following results were obtained.\smallskip\

Since $\inf |T_{\alpha}'| > 1$, applying Theorem 1 from the classical 1973 paper by Lasota and Yorke (cf.~\cite{LaY}, see also~\cite{LiY}) immediately yields the following assertion:
\begin{Proposition}\label{before ergodic}
If $\mu$ is an absolutely continuous invariant probability measure for $T_{\alpha}$, then there exists a function $h$ of bounded variation such that
$$
\mu (A) = \int_{A} h\, d \lambda, \,\, \lambda-\mbox{a.e.},\,\, \mbox{with $\lambda$ the Lebesgue measure},
$$
i.e.~ any absolutely continuous invariant probability measure has a version of its density function of bounded variation.
\end{Proposition}

We have the following result from~\cite{dJKN}.
\begin{Theorem}\label{ergodic}
Let $N \in \N_{\geq2}$. Then there is a unique absolutely continuous invariant probability measure $\mu_\alpha$ such that $\T$ is ergodic with respect to $\mu_\alpha$.
\end{Theorem}

With these results it was shown in~\cite{dJKN}, that if $|T_{\alpha}'(x)|>2$ for all $x\in I_{\alpha}$, there will be no gaps in $I_{\alpha}$.\medskip\

In this paper we will not focus on gaps, but rather on `plateaus' whether the entropy is constant. In~\cite{KL}, simulation of the entropy of $T_{\alpha}$ is given as a function of $\alpha\in (0,\sqrt{2}-1]$) in case $N=2$; see Figure~\ref{FigEntropySimulation}.

\begin{figure}[ht]
\includegraphics[width=10cm]{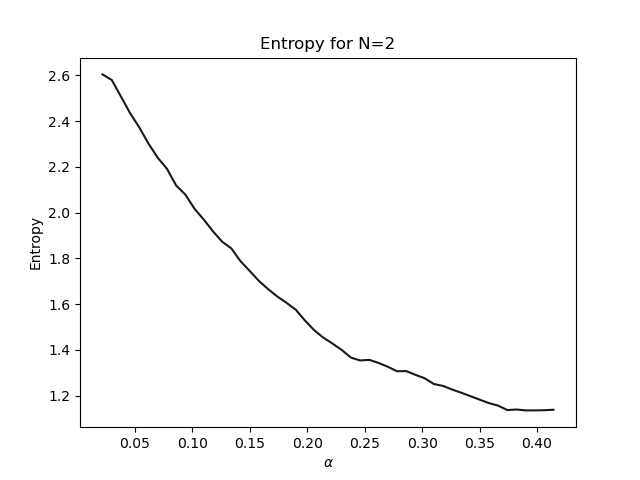}
\caption{A simulation of the entropy of $T_{\alpha}$ when $N=2$.}\label{FigEntropySimulation}
\end{figure}

In Figure~\ref{FigEntropySimulation} clearly a `plateau' of the entropy as a function of $\alpha$ is visible from $\frac{\sqrt{33}-5}{2}$ to $\sqrt{2}-1$ (which is the maximal possible value for $\alpha$ in case $N=2$; for larger values of $\alpha$ some of the digits would be equal to $0$). In~\cite{KL} it was then showed that for these values of $\alpha$ the so-called natural extensions could be built using a technique called \emph{quilting} (this technique will be explained in Section~\ref{sec:quilting}), and it could be shown that for $\alpha\in \left(\frac{\sqrt{33}-5}{2},\sqrt{2}-1\right)$ these natural extensions are metrically isomorphic. In general, a \emph{natural extension} is an almost surely minimal invertible system which has the original system (in this case $(I_{\alpha}, \mathcal{B}_{\alpha},\mu_{\alpha},T_{\alpha})$) as a factor. For continued fractions, the natural extension is (isomorphic to) some planar domain $\Omega_{\alpha}$, with an almost surely invertible map ${\mathcal T}_{\alpha}:\Omega_{\alpha}\to\Omega_{\alpha}$, given in the particular case of $N$-expansions by
\begin{equation}\label{eq:DefinitionNaturalExtensionMap}
{\mathcal T}_{\alpha}(x,y) = \left (T_{\alpha}(x),\frac{N}{d(x)+y}\right),
\end{equation}
where $d(x)\in\N$ is such, that $T_{\alpha}(x)=\frac{N}{x}-d(x)\in I_{\alpha}=[\alpha,\alpha +1)$ (cf.~(\ref{defNalphaGaussMap})). In~\cite{KL} the following result was obtained.

\begin{Theorem}\label{thm:natexonint}
For $\alpha\in\left[ \frac{\sqrt{33}-5}{2}, \sqrt{2}-1\right]$ the natural extension can be build (see Figure~\ref{Fig:NatExtN=2} below). Moreover, the invariant density $f_{\alpha}$ is given by:
\begin{eqnarray*}
f_{\alpha}(x) &=& H \Big(\frac{D}{2+Dx}\textbf{1}_{(\alpha,T(\alpha+1))} +\frac{E}{2+Ex}\textbf{1}_{(T(\alpha+1),T^2(\alpha))}+\frac{F}{2+Fx}\textbf{1}_{(T^2(\alpha),\alpha+1)}\\
&& -\frac{A}{2+Ax}\textbf{1}_{(\alpha,T^2(\alpha+1))} - \frac{B}{2+Bx}\textbf{1}_{(T^2(\alpha+1),T(\alpha))} -\frac{C}{2+Cx}\textbf{1}_{(T(\alpha),\alpha+1)} \Big), \\
\end{eqnarray*}
where $A=\frac{\sqrt{33}-5}{2}, B=\sqrt{2}-1, C=\frac{\sqrt{33}-3}{6},  D=2\sqrt{2}-2, E=\frac{\sqrt{33}-3}{2}, F=\sqrt{2}$ and $H^{-1}=\log\left(\frac{1}{32}(3+2\sqrt{2})(7+\sqrt{33})(\sqrt{33}-5)^2\right)\approx 0.25$ the normalizing constant.
\end{Theorem}

\begin{figure}[!htb]
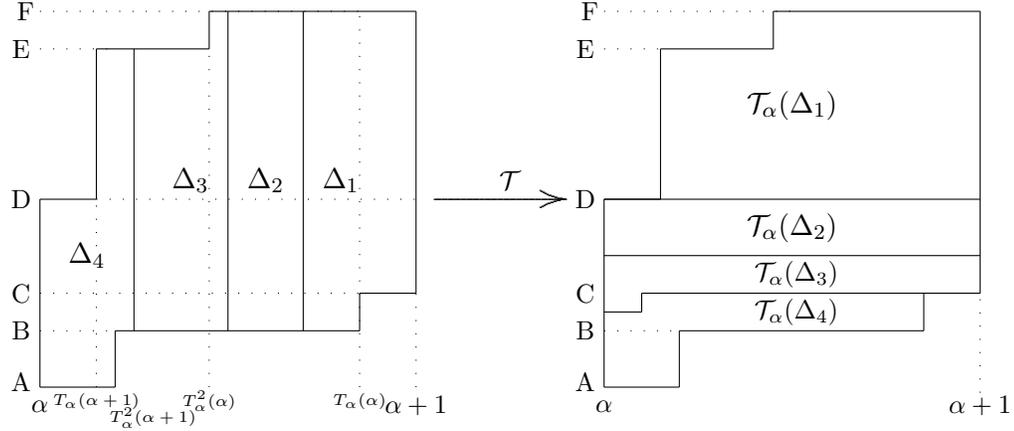

$$
\beginpicture
\setcoordinatesystem units <0.25cm,0.25cm>
\setplotarea x from -7 to 52, y from -1 to 22

\putrule from 0 0 to 4 0		
\putrule from 4 0 to 4 3
\putrule from 4 3 to 17 3
\putrule from 17 3 to 17 5
\putrule from 17 5 to 20 5
\putrule from 20 5 to 20 20
\putrule from 9 20 to 20 20
\putrule from 9 18 to 9 20
\putrule from 3 18 to 9 18
\putrule from 3 10 to 3 18
\putrule from 0 10 to 3 10
\putrule from 0 0 to 0 10

\putrule from 14 3 to 14 20
\putrule from 10 3 to 10 20
\putrule from 5 3 to 5 18

\arrow<10pt> [0.2,0.67] from 21 10 to 28 10
\put {$\mathcal{T}$} at 25 11

\putrule from 30 0 to 34 0		
\putrule from 34 0 to 34 3
\putrule from 34 3 to 47 3
\putrule from 47 3 to 47 5
\putrule from 47 5 to 50 5
\putrule from 50 5 to 50 20
\putrule from 39 20 to 50 20
\putrule from 39 18 to 39 20
\putrule from 33 18 to 39 18
\putrule from 33 10 to 33 18
\putrule from 30 10 to 33 10
\putrule from 30 0 to 30 10

\putrule from 30 10 to 50 10
\putrule from 30 7 to 50 7
\putrule from 32 5 to 50 5
\putrule from 32 4 to 32 5
\putrule from 30 4 to 32 4

\put {A} at -1 0.25
\put {F} at -0.75 20
\put {B} at -1 3
\put {C} at -1 5
\put {D} at -1 10
\put {E} at -1 18

\put {$\alpha$} at 0 -1
\put {$\alpha+1$} at 20 -1
\put {$\alpha$} at 30 -1
\put {$\alpha+1$} at 50 -1
\put {\tiny $T_{\alpha}(\alpha+1)$} at 3 -0.75
\put {\tiny $T_{\alpha}^2(\alpha+1)$} at 6 -1.6
\put {\tiny $T_{\alpha}(\alpha)$} at 17 -0.75
\put {\tiny $T_{\alpha}^2(\alpha)$} at 9 -0.75

\put {A} at 29 0.25
\put {F} at 29.25 20
\put {B} at 29 3
\put {C} at 29 5
\put {D} at 29 10
\put {E} at 29 18
\put {\large $\Delta_1$} at 16 11
\put {\large $\Delta_2$} at 12 11
\put {\large $\Delta_3$} at 8 11
\put {\large $\Delta_4$} at 2.5 7

\put {\large $\mathcal{T}_{\alpha}(\Delta_1)$} at 40 15
\put {\large $\mathcal{T}_{\alpha}(\Delta_2)$} at 40 8.5
\put { $\mathcal{T}_{\alpha}(\Delta_3)$} at 40 6
\put { $\mathcal{T}_{\alpha}(\Delta_4)$} at 40 4

\setdots
\putrule from 0 3 to 4 3
\putrule from 0 5 to 18 5
\putrule from 3 10 to 20 10
\putrule from 0 20 to 17 20
\putrule from 0 18 to 4 18

\putrule from 50 -1 to 50 5

\putrule from 30 3 to 34 3
\putrule from 30 18 to 33 18
\putrule from 30 20 to 40 20

\putrule from 3 7.5 to 3 10
\putrule from 3 0 to 3 6

\putrule from 9 0 to 9 18
\putrule from 17 0 to 17 20
\putrule from 20 0 to 20 5
\endpicture
$$
\caption{A planar natural extension in case $N=2$ for $\alpha\in\left( \frac{\sqrt{33}-5}{2}, \sqrt{2}-1\right)$.}\label{Fig:NatExtN=2}

\end{figure}

In this paper we will show, that for every integer $N\geq 2$ such plateaux exist, and give them explicitly. The number of such plateaux will be a function of $N$.

\section{A certain set of $\alpha$ for which there is matching for $T_{\alpha}$ in 3 steps for every $N\in\N$, $N\geq 2$, and a related planar domain for ${\mathcal T}_{\alpha}$}\label{sec:matching}
Let $N\in\N$, $N\geq 2$ arbitrary but fixed, and let\footnote{Usually we suppress in our notation the dependence of the various maps and domains on $N$.} $0<\alpha \leq N-1$. Let $T_{\alpha}: [\alpha ,\alpha +1]\to [\alpha , \alpha +1)$ be the Gauss map defined as in~(\ref{def:NGaussMap}), where $\{ d,d+1,\dots,d+i\}$ is the (finite) set of partial quotients (i.e.\ digits) for $T_{\alpha}$ (in~\cite{dJKN} is was shown that $i\in\N$, $i\geq 1$, so the number of partial quotients $i+1$ is at least 2).

Consider the partition ${\mathcal P} = \bigcup I_k$ of $[\alpha,\alpha+1]$, where $I_k=\{x\, |\, d_1(x)=k\}$. Note that we choose the largest digit $d+i$ always in such a way, that each partition element is an interval (see also~\cite{dJKN} for a small discussion about this).

In~\cite{dJKN} it was mentioned (cf.~Lemma~1 in~\cite{dJKN}), that given $N$ and $\alpha$, and $d(\alpha )$ as the largest possible digit of $T_{\alpha}$, one has that
$$
d(\alpha )\geq N-1\,\, {\text{if and only if}}\,\, \alpha<1.
$$
We have the following, similar result on the smallest possible digit $d=d(\alpha +1)$ of $T_{\alpha}$.

\begin{Lemma}\label{lemma1}
Let $N\in\mathbb{N}$, $N\geq 2$, and $0 < \alpha \leq \sqrt{N}-1$, then for the smallest possible digit $d=d(\alpha +1)$ of $T_{\alpha}$ we have that $d\in \{ 1,2,\dots,N-1\}$ and $\lim\limits_{\alpha\downarrow 0} d = N-1$.
\end{Lemma}

\begin{proof}
From $\alpha < N/(\alpha + 1) - d$, it follows that $\alpha^2 + (d + 1)\alpha + d - N < 0$. Since $\alpha,d>0$ it follows that $d<N$.
Furthermore, if $\alpha$ tends to $0$ it follows that $d$ tends to $N-1$. Note that if $\alpha = 0$, we have that $d=N$ (cf.~\cite{DKW}).
\end{proof}

The following result gives bounds on the number $i+1$ of possible digits.

\begin{Lemma}\label{lemma2}
For all $N\in\N$, $N\geq 2$, and $0 < \alpha \leq \sqrt{N}-1$, $d\geq 1$, one has $\frac{d}{\alpha}\leq i < \frac{d+1}{\alpha}+2$, where $i+1$ is the number of possible digits. Furthermore, $\lim\limits_{\alpha\downarrow 0} i = +\infty$.
\end{Lemma}

\begin{proof} ($i$)\, Since $T_{\alpha}$ is a map from $[\alpha, \alpha +1]$ to $[\alpha, \alpha +1)$ we have that $\alpha \leq N/(\alpha + 1) - d$ (which is the same as saying that $T_{\alpha}(\alpha +1)\geq \alpha$), it follows that $(\alpha + 1)(\alpha + d) \leq N $;
and from $N/\alpha - (d+ i) \leq \alpha + 1$, one trivially has that $N/\alpha - (\alpha + 1) \leq d + i$. Then
$(\alpha + 1)(\alpha + d)/\alpha -(\alpha + 1) \leq d + i$, and one has that $\alpha+(d+1)+d/\alpha-\alpha - 1\leq d+i$,
yielding that $i\geq d/\alpha$.

($ii$)\, Again since $T_{\alpha}$ is a map from $[\alpha, \alpha +1]$ to $[\alpha, \alpha +1)$ we have that  $N/(\alpha + 1) - d < \alpha + 1$; one immediately sees that $N < (\alpha + 1)(\alpha + d + 1)$; from $\alpha \leq N/\alpha -(d+ i)$, it immediately follows that $(d + i) \leq N/\alpha - \alpha$. Combining this yields that $d + i < (\alpha + 1)(\alpha + d + 1)/\alpha - \alpha = d+2+(d+1)/\alpha$; we find that $i < (d+1)/\alpha+2$.
\end{proof}

Now define ${\mathcal A}_{N,d,i}$ be the set of all $\alpha\in (0,\sqrt{N}-1]$ with digit set $\{d,d+1,\dots,d+i\}$. Furthermore, we define the sets $X_{N,d,i}$ and $X_{N,d,i,k}$ as follows:
\begin{eqnarray}\label{def:X(N,d,i)}
X_{N,d,i} &=& \{\alpha\in {\mathcal A}_{N,d,i}\, \Big{|}\, T_{\alpha}(\alpha)\in I_d^o, T_{\alpha}(\alpha+1)\in I_{d+i}^o\};\\
X_{N,d,i,k} &=& \{\alpha\in X_{N,d,i}\, \Big{|}\, T_{\alpha}^{2}(\alpha)\in I_k, T^2_{\alpha}(\alpha+1)\in I_{k+1}\},\quad \text{for $k=d,\dots,d+i-1$}.\label{def:X(N,d,i,k)}
\end{eqnarray}
Due to the fact that $\left| T_{\alpha}'(x)\right| >1$ for $x\in [\alpha,\alpha+1)$ we have that
$$
X_{N,d,i}=\left\{\alpha\in \Omega_{N,d,i}\,\, \Big{|}\,\, \frac{N}{d + 1 + \alpha} < \frac{N}{\alpha} - (d +i) < \alpha + 1, \,\,\, \alpha < \frac{N}{\alpha+1} - d < \frac{N}{d + i + \alpha} \right\}.
$$
In the next theorem we show that for $N\in\N$, $N\geq 2$, for which there exist positive integers $d$ and $i$ such that $N=\frac{d(d+i)}{i-1}$, for $\alpha\in{X_{N,d,i}}$ the corresponding maps $T_{\alpha}$ \emph{synchronize}\footnote{In many recent papers this property is called \emph{matching}.} in 3 steps; $T_{\alpha}^3(\alpha) = T_{\alpha}^3(\alpha +1)$. This property is key for us, as it helps us to construct the \emph{natural extensions} of the dynamical systems $([\alpha,\alpha +1), T_{\alpha})$, but also to understand why for such values of $\alpha$ the entropy is constant. At first it might not be clear that for every integer $N\geq 2$ positive integers $i$ and $d$ exist for which $N=\frac{d(d+i)}{i-1}$; this will be investigated in Proposition~\ref{prop:AtLeast3Solutions}.

\begin{Proposition}\label{prop:AtLeast3Solutions}
Let $N\geq 2$ be an integer, and let $D(N)$ be the number of pairs of integers $(d,i)$ with $d\geq 1$, $i\geq 2$, and $N = \frac{d(d+i)}{i-1}$.
Then if $N=2$ we have that $D(N)=1$, if $N=3,4$ we have that $D(N)=2$, and for $N\geq 5$ we have that $D(N)$ is at least three and at most $M(N) = (\sigma_0(N)-1)(\sigma_0(N+1)-1)$. Here $\sigma_0:\N\to\N$ is the divisor function of $n\in\N$, defined by:
\begin{equation}\label{def:divisorfunction}
\sigma_0(n) = \sum_{d|n} d^0,\quad \text{for $n\in\N$}.
\end{equation}
\end{Proposition}

\begin{proof}
Note that $N = \frac{d(d+i)}{i-1}$ can be rewritten as $i = \frac{d^2+N}{N-d}$. Setting for convenience $k = N - d$, we see that we have that
\begin{equation}\label{eq:PossibleValuesOfi}
i = \frac{d(d+1)}{k} + 1.
\end{equation}
Now we let $k$ run from 1 to $N-1$; clearly for at most $N-1$ values of $k$ we have that $i\in\Z$, but obviously this upper bound is far from sharp. If $d=N-1$ we have $k=1$, and obviously we have $i=(N-1)N+1\in\N$. If $d=N-2$ we have $k=2$, and obviously $i$ from~(\ref{eq:PossibleValuesOfi}) is an integer which is at least 2.

In case $N$ is \emph{even}, then one easily checks that if $d=\frac{N}{2}$, $i=\frac{N+4}{2}$, one has that $N = \frac{d(d+i)}{i-1}$. In case $N$ is odd, one again checks very easily that if $d=\frac{N-1}{2}$, $i=\frac{N+1}{2}$, one has that $N = \frac{d(d+i)}{i-1}$.

If $N=2$, we obviously only have as the only positive solution to~(\ref{eq:PossibleValuesOfi}): $(d=1,i=3)$. In case $N=3$ we have only 2 solutions: $(d=1,i=4)$ and $(d=2,i=7)$. For $N=3$ we have $\frac{N+1}{2}=2$, so there is no third solution. In case $N=4$ we again have 2 solutions: $(d=2,i=4)$ and $(d=3,i=13)$. In this case we have that $\frac{N}{2}=2$, which we already saw a a solution for $d$. In case $N\geq 5$ we indeed have at least 3 solutions: in case $N$ is even we have that $\frac{N}{2}<N-2$ (so $N>4$), and in case $N$ is odd we have $\frac{N-1}2<N-2$ (so $N>3$).

If we substitute $d=N-k$ in~(\ref{eq:PossibleValuesOfi}), we trivially find that:
\begin{equation}\label{eq:PossibleValuesOfi-2}
i = \frac{N(N+1)}{k} - 2N + k.
\end{equation}
Now for $N\in\N$ we have that $N$ and $N+1$ are relative prime, so we have that $\sigma_0(N(N+1)) = \sigma_0(N)\sigma_0(N+1)$ as $\sigma_0$ is an arithmetic function. Note that $k$ cannot be $N$ nor $N+1$ (in the first case we would have that $d=0$, and in the second case even $d=-1$; these are both impossible since digits $d$ are at least $1$), so we find from~(\ref{eq:PossibleValuesOfi-2}) and the fact that $i\in\N$, $i\geq 2$, that $D(N)$ is at most the number of divisors $k\in\{ 1,2,\dots,N-1\}$ of $N(N+1)$ for which $i$ from~(\ref{eq:PossibleValuesOfi-2}) is an integer at least 2.
\end{proof}
\medskip\

\begin{figure}[htb]
\minipage{0.5\textwidth}
\center{\includegraphics[trim=4cm 3cm 1cm 4cm, width=40mm]{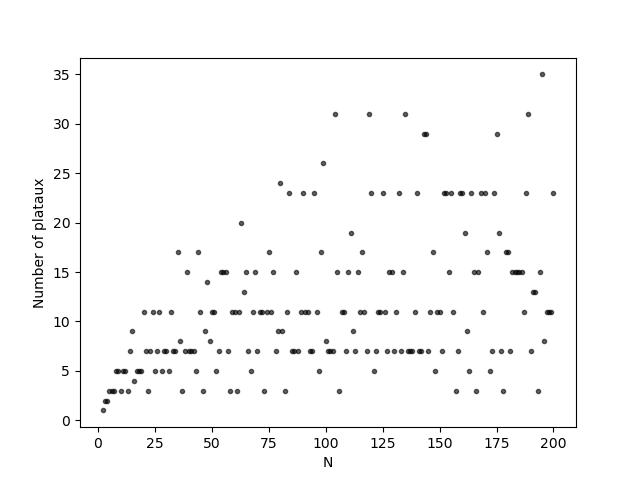}}

\endminipage
\minipage{0.5\textwidth}
\center{\includegraphics[trim=4cm 4cm 1cm 4cm, width=40mm]{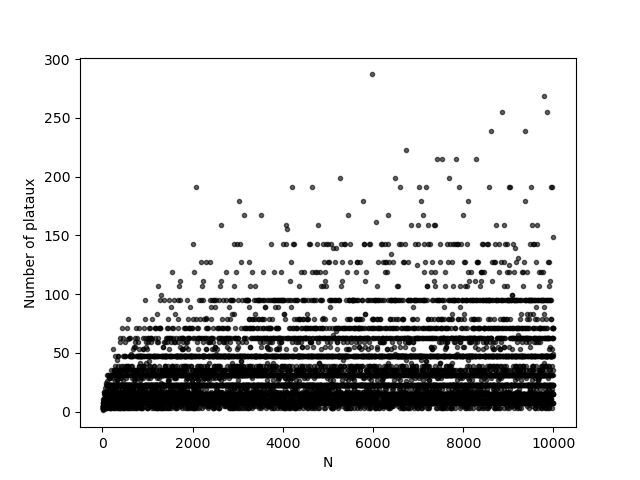}}
\endminipage
\vspace{1cm}\caption{$D(N)$ for $N=2,\dots,200$ (left) and $N=2,\dots,10.000$ (right) } \label{fig:D(N)forN=200andN=10000}
\end{figure}

\begin{Remarks}\label{rem:RemarkAboutDivisors}{\rm
($i$) For every real or complex $x$ one can define the \emph{sum of positive divisors function} $\sigma_x$ as
$$
\sigma_x(n) = \sum_{d|n} d^x,\quad \text{for $n\in\N$}.
$$
In this paper we are only interested in $x=0$, but e.g.\ $x=1$ yields the sum of all positive divisors of $n$, and $s(n)=\sigma_1(n)-1$ is the so-called \emph{aliquot sum}, i.e.\ the sum of all proper divisors of $n\in\N$. Obviously we have $\sigma_0(p)=2$ for all prime numbers $p$, and therefore $\liminf_{n\to\infty} \sigma_0(n) = 2$. On the other hand, it was shown by Severin Wigert (cf.~\cite{HW}, pp.\ 342–347, Section~18.1) that
$$
\limsup_{n\to\infty} \frac{\log \sigma_0(n)}{\log n/\log\log n} = \log 2.
$$
($ii$) Clearly $D(N)=3$ if $N\geq 5$ and $(N+1)/2$ are prime, or if $N/2$ and $N+1$ are prime (e.g.\ for $N=5$, $6$, $7$, $10$, $22$, $37$, $58$, $61$, $73$, $82$, $157, \dots$, $613,\dots$). The right-hand side figure in Figure~\ref{fig:D(N)forN=200andN=10000} seems to suggest that ${\displaystyle \liminf_{n\to\infty}} D(n) = 3$, but we have no proof of this.\\
In general $D(N)$ is smaller than the maximum possible value $M(N)= (\sigma_0(N)-1)(\sigma_0(N+1)-1)$. The smallest $N$ for which this happens is $N=8$; one easily sees that $M(8)=6$, while $D(8)=5$ The reason is, that $4|8$ and $3|9$, so $k=4\times 3=12|N(N+1)=72$, but $k=12\geq N=8$, and therefore we cannot find an admissible digit $d$ (since $d=N-k=8-12=-4$, and we must have $d\geq 1$). See also Figure~\ref{fig:D(N)forN=200andN=10000}, where we display $D(N)$ in the left-hand figure for $N=2,\dots,200$, and in the right-hand figure for $N=2,\dots,10.000$.}\hfill $\triangle$
\end{Remarks}

\begin{Theorem}\label{thm:MatchingIn3Steps}
Let $N\geq 2$ be an integer, and let $d, i\in\N$, $i\geq 2$, be such, that $N=\frac{d(d+i)}{i-1}$. Then for any $\alpha\in{X_{N,d,i}}$, one has that $T_{\alpha}^2(\alpha )\in I_k$ and $T_{\alpha}^2(\alpha +1)\in I_{k+1}$ for some $k\in\{ d,\dots,d+i-1\}$. Moreover, $T_{\alpha}^{3}(\alpha)=T_{\alpha}^{3}(\alpha+1)$.
\end{Theorem}

\begin{proof}
By definition of $X_{N,d,i}$ and $T_{\alpha}$, one has for $\alpha\in X_{N,d,i}$ that $T_{\alpha}^{2}(\alpha)=\frac{N}{\frac{N}{\alpha} - (d +i)} - d$, and that $T_{\alpha}^{2}(\alpha+1)=\frac{N}{\frac{N}{\alpha+1} - d} - (d +i)$.  Then,
\begin{eqnarray*}
\frac{N}{T_{\alpha}^{2}(\alpha)}&=&\frac{N}{\frac{N}{\frac{N}{\alpha} - (d +i)} - d} = -\frac{N(N-(d+i)\alpha)}{Nd - (d^2 + di + N)\alpha},\\
\frac{N}{T_{\alpha}^{2}(\alpha+1)}&=&\frac{N}{\frac{N}{\frac{N}{\alpha + 1} - d} - (d +i)} = -\frac{N(N-d(\alpha + 1))}{(d + i - \alpha -1)N - d(d + i)(\alpha +1))},
\end{eqnarray*}
and using CAS\footnote{CAS is an abbreviation of \emph{computer algebra system}.} one easily finds that:
$$
\frac{N}{T^{2}(\alpha)}-\left( \frac{N}{T^{2}(\alpha+1)}-1\right) = \left( d^2 +di - N(i - 1)\right) \cdot R_{N,d,i,\alpha},
$$
where $R_{N,d,i,\alpha}$ satisfies:
$$
R_{N,d,i,\alpha} = \frac{\big((\alpha^2 + \alpha)d^2 + ((-2\alpha - 1)N + d i\alpha(\alpha + 1)) + (N - \alpha(i - \alpha - 1))N\big)}{(-d^2\alpha + (-i\alpha + N)d - N\alpha)((-\alpha - 1)d^2 + (-i\alpha + N - i)d + N(i - \alpha - 1))}.
$$

Note that if $d^2+di-N(i-1)=0$, so if $N=\frac{d(d+i)}{i-1}$, we have that:
$$
\frac{N}{T_{\alpha}^2(\alpha )} = \frac{N}{T_{\alpha}^2(\alpha +1)} -1.
$$
Since the length of the interval $[\alpha ,\alpha +1)$ is 1, we see that for $N=\frac{d(d+i)}{i-1}$ we have \emph{matching} in 3 steps: $T_{\alpha}^{3}(\alpha)=T_{\alpha}^{3}(\alpha+1)$. Furthermore, $T_{\alpha}^2(\alpha )\in I_k$ and $T_{\alpha}^2(\alpha +1)\in I_{k+1}$ for some $k\in\{ d,\dots,d+i-1\}$.
\end{proof}

\begin{Remarks}\label{rem:RemarkAboutTheMatchingTheorem}{\rm
($i$) For the case $N=2$ the result of Theorem~\ref{thm:MatchingIn3Steps} were already obtained in~Theorem~3.1 of~\cite{KL}.

($ii$) Note that under the conditions of Theorem~\ref{thm:MatchingIn3Steps}, an immediate consequence of the proof of Theorem~\ref{thm:MatchingIn3Steps} is that
$$
X_{N,d,i} = \bigcup_{k=d}^{d+i-1} X_{N,d,i,k}.
$$
($iii$) The conditions of Theorem~\ref{thm:MatchingIn3Steps} which lead to matching in 3 steps were obtained using an extensive search using CAS. We did not find other relations, but that obviously does not imply these do not exist; see for a brief discussion Example~\ref{ex:entropy}.

($iv$) In the definition~(\ref{def:X(N,d,i)}) of $X_{N,d,i}$ we demanded that $T_{\alpha}(\alpha)\in I_d^o$, and that $T_{\alpha}(\alpha+1)\in I_{d+i}^o$; the reason is, that we must avoid the endpoints of the cylinders $I_d$ and $I_{d+i}$ in order to be able to draw the conclusions of Theorem~\ref{thm:MatchingIn3Steps}. For example, if $T_{\alpha}(\alpha)=\alpha +1$ and $T_{\alpha}(\alpha +1) = \alpha$, clearly all branches of $T_{\alpha}$ are full and there will be no matching in 3 steps (or any number of steps), but it is very easy to construct the natural extension; see Theorem 1 in~\cite{dJKN, S} where it is explicitly stated for which $N$ and $\alpha$ one has that $T_{\alpha}$ only has full branches, and also~\cite{S} where the density of $T_{\alpha}$ (and of its dual map) is given. Although not explicitly stated, the driving idea behind these calculations is the concept of \emph{planar natural extension}. Another example is, when $N, d, i$ and $\alpha$ are such, that
$$
T_{\alpha}(\alpha) = \frac{N}{\alpha + 1+d},\quad \text{or $T_{\alpha}(\alpha +1)=\alpha$}.
$$
Here $N/(\alpha + 1 +d)$ is the left endpoint of the cylinder $I_d$ and the right endpoint of the cylinder $I_{d+1}$. Note that by definition of the map $T_{\alpha}$ we have that $N/(\alpha + 1 +d)\in I_{d+1}$. So formally, $N/(\alpha + 1 +d)\not\in I_d$, so certainly $N/(\alpha + 1 +d)\not\in I_d^o$, and therefore $\alpha\not\in X_{N,d,i}$ (so clearly the statement of Theorem~\ref{thm:MatchingIn3Steps} does not apply to this $\alpha$. Indeed we don't have matching in 3 steps: $T_{\alpha}(\alpha ) =  N/(\alpha + 1 +d)$, and therefore $T_{\alpha}^2(\alpha )=\alpha = T_{\alpha}(\alpha +1)$. Still, this is an interesting case, as it is very easy to construct the planar natural extension, and once this has been obtained, to find the ${\mathcal T}_{\alpha}$-invariant measure, and by projecting the invariant measure of the dual algorithm. As the proof of Theorem~\ref{thm:PlanarDomainNaturalExtension} is similar to this construction, we decided to skip it here, but discuss this case briefly in Section~\ref{sec:entropy}. Note however, that for this example the dual algorithm exist, while in general (under the conditions of Theorem~\ref{thm:MatchingIn3Steps}) there is no dual algorithm; see also~\cite{P}, p.~58.}\hfill $\triangle$
\end{Remarks}

Under the assumptions of Theorem~\ref{thm:MatchingIn3Steps} and using the matching in 3 steps guaranteed by Theorem~\ref{thm:MatchingIn3Steps}, we will show how we can find the planar domains $\Omega_{\alpha}$ of the natural extensions of the dynamical systems $(I_{\alpha}, T_{\alpha})$ for $\alpha\in X_{N,d,i}$. Recall the definition of the map ${\mathcal T}_{\alpha}: \Omega_{\alpha}\to \Omega_{\alpha}$ from~(\ref{eq:DefinitionNaturalExtensionMap}):
$$
{\mathcal T}_{\alpha}(x,y) = \left( \frac{N}{x}-d(x), \frac{N}{d(x)+y}\right),
$$
where $x\in I_{\alpha}$, $d(x)\in\{d,d+1,\cdots,d+i\}$ and $N,d,i$ positive integers, with $N, i\geq 2$.

\begin{Theorem}\label{thm:PlanarDomainNaturalExtension}
Let $N\geq 2$ be an integer, and let $d\geq 1$ and $i\geq 2$ be integers, such that $N=\frac{d(d+i)}{i-1}$. Let $\alpha\in X_{N,d,i}$ arbitrary, and let the planar domain $\Omega_{\alpha}$ be the polygon, bounded by the straight line segments between the vertices (in clockwise order) $(\alpha , A)$, $(T_{\alpha}^2(\alpha +1), A)$, $(T_{\alpha}^2(\alpha +1), B)$,
$(T_{\alpha}(\alpha ), B)$, $(T_{\alpha}(\alpha ), C)$, $(\alpha +1, C)$, $(\alpha +1, F)$, $(T_{\alpha}^2(\alpha ),F)$, $(T_{\alpha}^2(\alpha ), E)$, $(T_{\alpha}(\alpha +1),E)$, $(T_{\alpha}(\alpha +1),D)$, $(\alpha ,D)$, and finally `back' to $(\alpha ,A)$ (see also Figures~\ref{fig:OurNaturalExtensions1} and~\ref{fig:OurNaturalExtensions2} where $\Omega_{\alpha}$ is illustrated for various $\alpha$), where $0 < A < B < C < D < E < F$.

Then if the map ${\mathcal T}_{\alpha}: \Omega_{\alpha}\to\Omega_{\alpha}$ is bijective almost surely with respect to Lebesgue measure $\lambda$  we have that
\begin{equation}\label{eq:ValuesOfTheVariusHeights1}
A = E -1 = \frac{-(d+i+1)+\sqrt{(d+i+1)^2 + 4N}}{2},\quad B = F -1 = \frac{-(d+1)+\sqrt{(d-1)^2+4N}}{2},
\end{equation}
and
\begin{equation}\label{eq:ValuesOfTheVariusHeights2}
C = \frac{N(-(d+i-1) + \sqrt{(d+i+1)^2 + 4N})}{2(d+i+N)},\quad D = \frac{N(-(d+1)+\sqrt{(d-1)^2+4N})}{2(N-d)},
\end{equation}
and indeed we have that $0 < A < B < C < D < E < F$.
\end{Theorem}

\begin{proof}
As was the case in~\cite{KL} for $N=2$, we need to consider several cases, depending on the value of $k\in\{ d,d+1,\dots,d+i-1\}$ for which $\alpha\in X_{N,d,i,k}$. Since all these case are proved in a similar way, we only consider the case $k=d$ here; the other cases are left to the reader.

If $\alpha \in X_{N,d,i,k}$, we have by definition~(\ref{def:X(N,d,i,k)}) of $X_{N,d,i,k}$ that
$$
T_{\alpha}^2(\alpha ) \in I_d,\quad T_{\alpha}^2(\alpha +1)\in I_{d+1},
$$
(and since $\alpha \in X_{N,d,i}$ we also have (by definition~(\ref{def:X(N,d,i)})) that $T_{\alpha}(\alpha ) \in I_d$ and $T_{\alpha}(\alpha +1)\in I_{d+i}$), and there is matching, as $T_{\alpha}^3(\alpha )=T_{\alpha}^3(\alpha +1)$. We now will show that the polygon $\Omega_{\alpha}$ satisfies the various values of $A$, $B$, et cetera, as mentioned in~(\ref{eq:ValuesOfTheVariusHeights1}) and~(\ref{eq:ValuesOfTheVariusHeights2}).

Define the two-dimensional `cylinders' $\Theta_j$ as:
\begin{equation}\label{def:Thetaj}
\Theta_j := \{ (x,y)\in\Omega_{\alpha}\, |\, x\in\Delta_j\},\quad \text{for $j=d,d+1,\dots,d+i$},
\end{equation}
and recall from~\cite{KL} that we want the images of the various $\Theta_j$ under ${\mathcal T}_{\alpha}$ to `\emph{laminate}'; there should not be horizontal `gaps' between the images, as these will lead to infinitely many of such horizontal `gaps' (see~\cite{KL} for more details).
So we must choose $A$, $B$, et cetera in such a way, that for $j=d,d+1,\dots,d+i$ the polygon ${\mathcal T}_{\alpha}(\Theta_j)$ is mapped ``seamlessly on top'' of the polygon ${\mathcal T}_{\alpha}(\Theta_{j+1})$; see also Figures~\ref{fig:OurNaturalExtensions1} and~\ref{fig:OurNaturalExtensions2}, and also~\cite{KL} for more details. By definition of the second coordinate of the map ${\mathcal T}_{\alpha}$ this occurs when:
\begin{eqnarray}
&& A = \frac{N}{d+i+E},\quad B = \frac{N}{d+i+D},\quad C= \frac{N}{d+i+A},\quad C= \frac{N}{d+i-1+E},\label{eq:RelationsBetweenHeights1} \\
&& D =\frac{N}{d+1+B},\quad D = \frac{N}{d+F},\quad E = \frac{N}{d+C},\quad F = \frac{N}{d+B}.\label{eq:RelationsBetweenHeights2}
\end{eqnarray}
Finally, from Figures~\ref{fig:OurNaturalExtensions1} and~\ref{fig:OurNaturalExtensions2} we see we also need to have that
$$
M = \frac{N}{k+E} = \frac{N}{k+1+A},
$$
which is equivalent with $E = A + 1$; we will see below that this also follows from~(\ref{eq:RelationsBetweenHeights1}). Note that due to the fact there is matching the set ${\mathcal T}_{\alpha}(\Theta_k)$ has a `snug fit' on top of the set ${\mathcal T}_{\alpha}(\Theta_{k+1})$.

From~(\ref{eq:RelationsBetweenHeights1}) resp.~(\ref{eq:RelationsBetweenHeights2}) we see that:
$$
\frac{N}{d+i+A}= \frac{N}{d+i-1+E},\quad \text{resp. }\,\, \frac{N}{d+1+B} = \frac{N}{d+F},
$$
and it follows that $1+A=E$ resp.\ $1+B=F$. But then we immediately have from~(\ref{eq:RelationsBetweenHeights1}) resp.~(\ref{eq:RelationsBetweenHeights2}) that
\begin{equation}\label{eq:RelationsBetweenAandEandbetwwenBandF}
E-1 = \frac{N}{d+i+E},\quad \text{ and that}\,\, F = \frac{N}{d+F-1}.
\end{equation}
From the first equation in~(\ref{eq:RelationsBetweenAandEandbetwwenBandF}) we find a quadratic equation with determinant $(d+i+1)^2 + 4N>0$ and one positive root $E$:
$$
E=\frac{-(d+i-1)+\sqrt{(d+i+1)^2 + 4N}}{2},
$$
hence $A$ is also known (and positive). The second equation in~(\ref{eq:RelationsBetweenAandEandbetwwenBandF}) yields $F$ (and therefore also $B$) in a similar way:
$$
F = \frac{-(d-1)+\sqrt{(d-1)^2+4N}}{2}>0.
$$
Since we know $A$ (and $E$), from~(\ref{eq:RelationsBetweenHeights1}) also $C$ immediately follows. Similarly, $D$ immediately follows from~(\ref{eq:RelationsBetweenHeights2}) since $B$ (and $F$) are known. Due to Lemma~\ref{lemma1} we know that $d<N$, and therefore $(d-1)^2+4N>(d-1)^2+4d= (d+1)^2$, and we see that $D>0$. We still need to show that $A < B < C < D < E < F$. This can be proved by contradiction. Suppose e.g.\ that $A\geq B$. From~(\ref{eq:RelationsBetweenHeights1}) we immediately see that this is equivalent with $D\geq E$. From~(\ref{eq:RelationsBetweenHeights2}) we see that $D\geq E$ is equivalent with $C\geq F$, and from~(\ref{eq:RelationsBetweenHeights1}) and~(\ref{eq:RelationsBetweenHeights2}) it follows that $C\geq F$ is equivalent with $B\geq A+i$. Since $i\geq 2$, this last inequality immediately leads to $B\geq A+i\geq B+i>B$, which is impossible. So we must have that $A<B$ (and immediately that $D<E$). The other inequalities can be obtained in a similar way by contradiction. For example, from~(\ref{eq:RelationsBetweenHeights1}) and~(\ref{eq:RelationsBetweenHeights2}) we see that $B\geq C$ $\Leftrightarrow$ $A\geq D$ $\Leftrightarrow$ $F\geq E+i$. But also we have that $E\geq F$ $\Leftrightarrow$ $B\geq C$, so assuming $B\geq C$ we find that $F\geq E+i$ and that $E\geq F$, and again due to the fact that $i\geq 2$ we find $F\geq E+i \geq F+i>F$, which is impossible. The other two inequalities are left to the reader.

As ${\mathcal T}_{\alpha}: \Theta_j\to {\mathcal T}_{\alpha}(\Theta_j)$ is bijective for $j=d,d+1,\dots,d+i$, it is now follows that ${\mathcal T}_{\alpha}:\Omega_{\alpha}\to\Omega_{\alpha}$ is bijective almost surely.
\end{proof}

\begin{Remarks}\label{rem:TwoExpressionsForB}{\rm
Note that from~(\ref{eq:RelationsBetweenHeights1}) and~(\ref{eq:RelationsBetweenHeights2}) one finds that:
\begin{equation}\label{eq:BasCF}
B = \frac{N}{d+i + \displaystyle{\frac{N}{d+1+B}}},
\end{equation}
from which we find that
\begin{equation}\label{eq:OtherExpressionForB}
B = \frac{-(d+1)(d+i) + \sqrt{(d+1)^2(d+i)^2 + N(d+1)(d+i)}}{2(d+i)}.
\end{equation}
It is not immediately apparent why this last expression~(\ref{eq:OtherExpressionForB}) for $B$ is equal to the one in~(\ref{eq:ValuesOfTheVariusHeights1}). However, it is an easy exercise to see that they are equal if and only if $N=\frac{d(d+i)}{i-1}$. Although expression~(\ref{eq:OtherExpressionForB}) for $B$ is less attractive than the one from~(\ref{eq:ValuesOfTheVariusHeights1}), it comes in handy in Remark~\ref{rem:isomorphisms}($i$). }\hfill $\triangle$
\end{Remarks}

\begin{Example}\label{ex:ExampleN=2d=1i=3}{\rm
In Figures~\ref{Fig:NatExtN=2} and~\ref{fig:OurNaturalExtensions1} the cases $\alpha\in X_{N,d,i,k}$, for $k=1,2,3$, are illustrated when $N=2, d=1, i=3$. Note that in this case we have: $A=\frac{\sqrt{33}-5}{2}=E-1$, $B=\sqrt{2}-1=F-1$, $C=\frac{\sqrt{33}-3}{6}$ and $D=2\sqrt{2}-2$, which were also obtained on p.~121 in~\cite{KL}. Figure~\ref{fig:OurNaturalExtensions2} illustrates the cases $\alpha\in X_{d,i,k}$, for $k=1,2,3$, when $N=3$, $d=1, i=2$, and $d=2,i=7$.}\hfill $\triangle$
\end{Example}

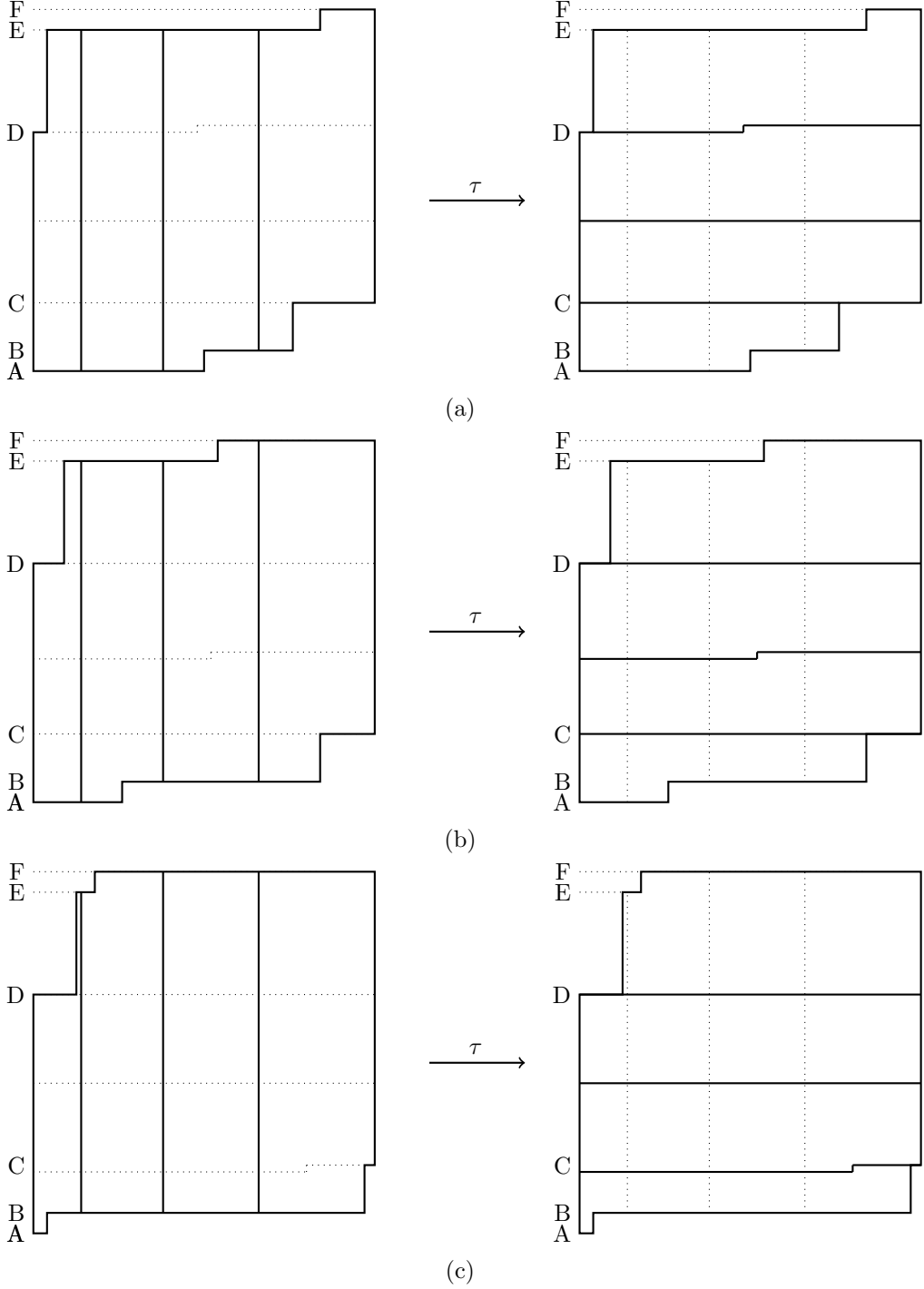
\begin{figure}[htp]
\begin{center}
\begin{tikzpicture}
\draw[black,  thick] (0,0) -- (2.5,0) --(2.5,0.3)-- (3.8,0.3) --(3.8,1)-- (5,1)--(5,5.3)--(4.2,5.3)--(4.2,5)--(0.2,5)--(0.2,3.5)--(0,3.5) -- cycle;
\draw[black, thick] (3.3,0.3) -- (3.3,5);
\draw[black, thick] (1.9,0) -- (1.9,5);
\draw[black, thick] (0.7,0) -- (0.7,5);

\filldraw[black] (0,0)   node[anchor=east]{ A};
\filldraw[black] (0,0.3)   node[anchor=east]{ B};
\filldraw[black] (0,1)   node[anchor=east]{ C};
\filldraw[black] (0,2.2)   node[anchor=east]{ };
\filldraw[black] (0,3.5)   node[anchor=east]{ D};
\filldraw[black] (0,5)   node[anchor=east]{ E};
\filldraw[black] (0,5.3)   node[anchor=east]{ F};

\draw[dotted, black] (0,1) -- (3.8,1);
\draw[dotted, black] (0,2.2) -- (5,2.2);
\draw[dotted, black] (2.4,3.5) -- (2.4,3.6);
\draw[dotted, black] (2.4,3.6) -- (5,3.6);
\draw[dotted, black] (0.2,3.5) -- (2.4,3.5);

\draw[ black,thick, ->] (5.8,2.5)-- (7.2,2.5) ;

\draw[dotted, black] (0,5) -- (0.2,5);
\draw[dotted, black] (0,5.3) -- (4.2,5.3);

\filldraw[black] (6.7,2.7)   node[anchor=east]{$\tau$ };

\draw[black,  thick] (0+8,0) -- (2.5+8,0) --(2.5+8,0.3)-- (3.8+8,0.3) --(3.8+8,1)-- (5+8,1)--(5+8,5.3)--(4.2+8,5.3)--(4.2+8,5)--(0.2+8,5)--(0.2+8,3.5)--(0+8,3.5) -- cycle;
\draw[dotted, black] (3.3+8,0.3) -- (3.3+8,5);
\draw[dotted, black] (1.9+8,0) -- (1.9+8,5);
\draw[dotted, black] (0.7+8,0) -- (0.7+8,5);
\filldraw[black] (0,0)   node[anchor=east]{ A};
\draw[black, thick] (0+8,1) -- (3.8+8,1);
\draw[black, thick] (0+8,2.2) -- (5+8,2.2);
\draw[black, thick] (2.4+8,3.5) -- (2.4+8,3.6);
\draw[black, thick] (2.4+8,3.6) -- (5+8,3.6);
\draw[black, thick] (0.2+8,3.5) -- (2.4+8,3.5);
\filldraw[black] (8+0,0)   node[anchor=east]{ A};
\filldraw[black] (8+0,0.3)   node[anchor=east]{ B};
\filldraw[black] (8+0,1)   node[anchor=east]{ C};
\filldraw[black] (8+0,2.2)   node[anchor=east]{ };
\filldraw[black] (8+0,3.5)   node[anchor=east]{ D};
\filldraw[black] (8+0,5)   node[anchor=east]{ E};
\filldraw[black] (8+0,5.3)   node[anchor=east]{ F};

\draw[dotted, black] (8+0,5) -- (8+0.2,5);
\draw[dotted, black] (8+0,5.3) -- (8+4.2,5.3);

\end{tikzpicture}

(a)

\begin{tikzpicture}
\draw[black,  thick] (0,0) -- (1.3,0) --(1.3,0.3)-- (3.8+0.4,0.3) --(3.8+0.4,1)-- (5,1)--(5,5.3)--(4.2-1.5,5.3)--(4.2-1.5,5)--(0.2+0.25,5)--(0.25+0.2,3.5)--(0,3.5) -- cycle;
\draw[black, thick] (3.3,0.3) -- (3.3,5+0.3);
\draw[black, thick] (1.9,0.3) -- (1.9,5);
\draw[black, thick] (0.7,0) -- (0.7,5);

\filldraw[black] (0,0)   node[anchor=east]{ A};
\filldraw[black] (0,0.3)   node[anchor=east]{ B};
\filldraw[black] (0,1)   node[anchor=east]{ C};
\filldraw[black] (0,2.2)   node[anchor=east]{ };
\filldraw[black] (0,3.5)   node[anchor=east]{ D};
\filldraw[black] (0,5)   node[anchor=east]{ E};
\filldraw[black] (0,5.3)   node[anchor=east]{ F};

\draw[dotted, black] (0,1) -- (5,1);
\draw[dotted, black] (0,2.1) -- (2.6,2.1);
\draw[dotted, black] (2.6,2.1) -- (2.6,2.2);
\draw[dotted, black] (2.6,2.2) -- (5,2.2);
\draw[dotted, black] (0,3.5) -- (5,3.5);

\draw[ black,thick, ->] (5.8,2.5)-- (7.2,2.5) ;

\draw[dotted, black] (0,5) -- (0.2+0.25,5);
\draw[dotted, black] (0,5.3) -- (4.2,5.3);

\filldraw[black] (6.7,2.7)   node[anchor=east]{$\tau$ };

\draw[black,  thick] (0+8,0) -- (1.3+8,0) --(1.3+8,0.3)-- (0.4+3.8+8,0.3) --(0.4+3.8+8,1)-- (5+8,1)--(5+8,5.3)--(4.2+8-1.5,5.3)--(4.2+8-1.5,5)--(0.25+0.2+8,5)--(0.25+0.2+8,3.5)--(0+8,3.5) -- cycle;
\draw[dotted, black] (3.3+8,0.3) -- (3.3+8,5+0.3);
\draw[dotted, black] (1.9+8,0+0.3) -- (1.9+8,5);
\draw[dotted, black] (0.7+8,0) -- (0.7+8,5);
\filldraw[black] (0,0)   node[anchor=east]{ A};
\draw[black, thick] (8+0,1) -- (8+5,1);
\draw[black, thick] (8+0,2.1) -- (8+2.6,2.1);
\draw[black, thick] (8+2.6,2.1) -- (8+2.6,2.2);
\draw[black, thick] (8+2.6,2.2) -- (8+5,2.2);
\draw[black, thick] (8+0,3.5) -- (8+5,3.5);

\filldraw[black] (8+0,0)   node[anchor=east]{ A};
\filldraw[black] (8+0,0.3)   node[anchor=east]{ B};
\filldraw[black] (8+0,1)   node[anchor=east]{ C};
\filldraw[black] (8+0,2.2)   node[anchor=east]{ };
\filldraw[black] (8+0,3.5)   node[anchor=east]{ D};
\filldraw[black] (8+0,5)   node[anchor=east]{ E};
\filldraw[black] (8+0,5.3)   node[anchor=east]{ F};

\draw[dotted, black] (8+0,5) -- (8+0.2+0.25,5);
\draw[dotted, black] (8+0,5.3) -- (8+4.2,5.3);

\end{tikzpicture}

(b)

\begin{tikzpicture}
\draw[black,  thick] (0,0) -- (0.2,0) --(0.2,0.3)-- (4.85,0.3) --(4.85,1)-- (5,1)--(5,5.3)--(0.9,5.3)--(0.9,5)--(0.63,5)--(0.63,3.5)--(0,3.5) -- cycle;
\draw[black, thick] (3.3,0.3) -- (3.3,5.3);
\draw[black, thick] (1.9,0.3) -- (1.9,5.3);
\draw[black, thick] (0.7,0.3) -- (0.7,5);

\filldraw[black] (0,0)   node[anchor=east]{ A};
\filldraw[black] (0,0.3)   node[anchor=east]{ B};
\filldraw[black] (0,1)   node[anchor=east]{ C};
\filldraw[black] (0,2.2)   node[anchor=east]{ };
\filldraw[black] (0,3.5)   node[anchor=east]{ D};
\filldraw[black] (0,5)   node[anchor=east]{ E};
\filldraw[black] (0,5.3)   node[anchor=east]{ F};

\draw[dotted, black] (0,0.9) -- (4,0.9);
\draw[dotted, black] (4,0.9) -- (4,1);
\draw[dotted, black] (4,1) -- (5,1);
\draw[dotted, black] (0,2.2) -- (5,2.2);
\draw[dotted, black] (0,3.5) -- (5,3.5);

\draw[ black,thick, ->] (5.8,2.5)-- (7.2,2.5) ;

\draw[dotted, black] (0,5) -- (0.63,5);
\draw[dotted, black] (0,5.3) -- (4.2,5.3);

\filldraw[black] (6.7,2.7)   node[anchor=east]{$\tau$ };

\draw[black,  thick] (8+0,0) -- (8+0.2,0) --(8+0.2,0.3)-- (8+4.85,0.3) --(8+4.85,1)-- (8+5,1)--(8+5,5.3)--(8+0.9,5.3)--(8+0.9,5)--(8+0.63,5)--(8+0.63,3.5)--(8+0,3.5) -- cycle;
\draw[dotted, black] (3.3+8,0.3) -- (3.3+8,5.3);
\draw[dotted, black] (1.9+8,0.3) -- (1.9+8,5.3);
\draw[dotted, black] (0.7+8,0.3) -- (0.7+8,5);
\filldraw[black] (0,0)   node[anchor=east]{ A};
\draw[black, thick] (8+0,0.9) -- (8+4,0.9);
\draw[black, thick] (8+4,0.9) -- (8+4,1);
\draw[black, thick] (8+4,1) -- (8+5,1);
\draw[black, thick] (8+0,2.2) -- (8+5,2.2);
\draw[black, thick] (8+0,3.5) -- (8+5,3.5);
\filldraw[black] (8+0,0)   node[anchor=east]{ A};
\filldraw[black] (8+0,0.3)   node[anchor=east]{ B};
\filldraw[black] (8+0,1)   node[anchor=east]{ C};
\filldraw[black] (8+0,2.2)   node[anchor=east]{ };
\filldraw[black] (8+0,3.5)   node[anchor=east]{ D};
\filldraw[black] (8+0,5)   node[anchor=east]{ E};
\filldraw[black] (8+0,5.3)   node[anchor=east]{ F};

\draw[dotted, black] (8+0,5) -- (8+0.63,5);
\draw[dotted, black] (8+0,5.3) -- (8+4.2,5.3);

\end{tikzpicture}

(c)
\caption{$\Omega_{\alpha}$ and $T(\Omega_{\alpha})$ with (a): $\alpha\in{X_{N,d,i,1}}$; (b): $\alpha\in{X_{N,d,i,2}}$; (c): $\alpha\in{X_{N,d,i,3}}$, for $N=2, d=1, i=3$.}
\label{fig:OurNaturalExtensions1}
\end{center}
\end{figure}

\begin{Theorem}\label{thm:IntervalOfEqualEntropy}
Let $N\geq 2$ be an integer, and let $d,i\in\N$, $i\geq 2$, such that $N=\frac{d(d+1)}{i-1}$. Then $X_{N,d,i}=(A,B)$, where $A$ and $B$ are from~(\ref{eq:ValuesOfTheVariusHeights1}).
\end{Theorem}

\begin{proof}
In the proof of Theorem~\ref{thm:PlanarDomainNaturalExtension} we saw in~(\ref{eq:RelationsBetweenHeights1}) that $A = \frac{N}{d+i+E}$, and we derived that $E = 1 + A$; these combined yield that:
\begin{equation}\label{eq:A1}
\frac{N}{A}-(d+i)=1+A,\quad \text{i.e.\ $T_{A}(A)=1+A$}.
\end{equation}
Furthermore, in the proof of Theorem~\ref{thm:PlanarDomainNaturalExtension} we saw in~(\ref{eq:RelationsBetweenHeights1}) that $E= \frac{N}{d+C}$, and from~(\ref{eq:RelationsBetweenHeights1}) we derived that $C= \frac{N}{d+i+A}$ and $E= 1+A$. These yield that
\begin{equation}\label{eq:A2}
\frac{N}{A+1}- d = \frac{N}{A+d+i},
\end{equation}
i.e.\ $T_{A}(A+1)$ is equal to the right endpoint of the cylinder $\Delta_{d+i}$. From~(\ref{eq:A1}) and~(\ref{eq:A2}) we see that $A$ is an endpoint of $X_{X,d,i}$, which (by definition~(\ref{def:X(N,d,i)}) of $X_{N,d,i}$) does not belong to $X_{N,d,i}$.

Also we saw in~(\ref{eq:RelationsBetweenHeights1}) that $B= \frac{N}{d+i+D}$, and we obtained that $D = \frac{N}{d+1+B}$, yielding that \begin{equation}\label{eq:B1}
\frac{N}{B+d+1}=\frac{N}{B}-(d+i),\quad \text{i.e.\ $T_{B}(B)$ is the left endpoint of the cylinder $\Delta_d$}.
\end{equation}
In~(\ref{eq:RelationsBetweenHeights2}) we also saw that $F = \frac{N}{d+B}$, and we found that $F = 1+B$.
From these we see that
\begin{equation}\label{eq:B2}
B=\frac{N}{B+1}- d,\quad \text{i.e.\ $T_{B}(B)=1+B$}.
\end{equation}
Now~(\ref{eq:B1}) and~(\ref{eq:B2}) yield that $B$ is an endpoint of $X_{X,d,i}$, which (again by definition~(\ref{def:X(N,d,i)}) of $X_{N,d,i}$) does not belong to $X_{N,d,i}$.

Since $A<B$, it follows that $(A,B) = X_{N,d,i}$.
\end{proof}

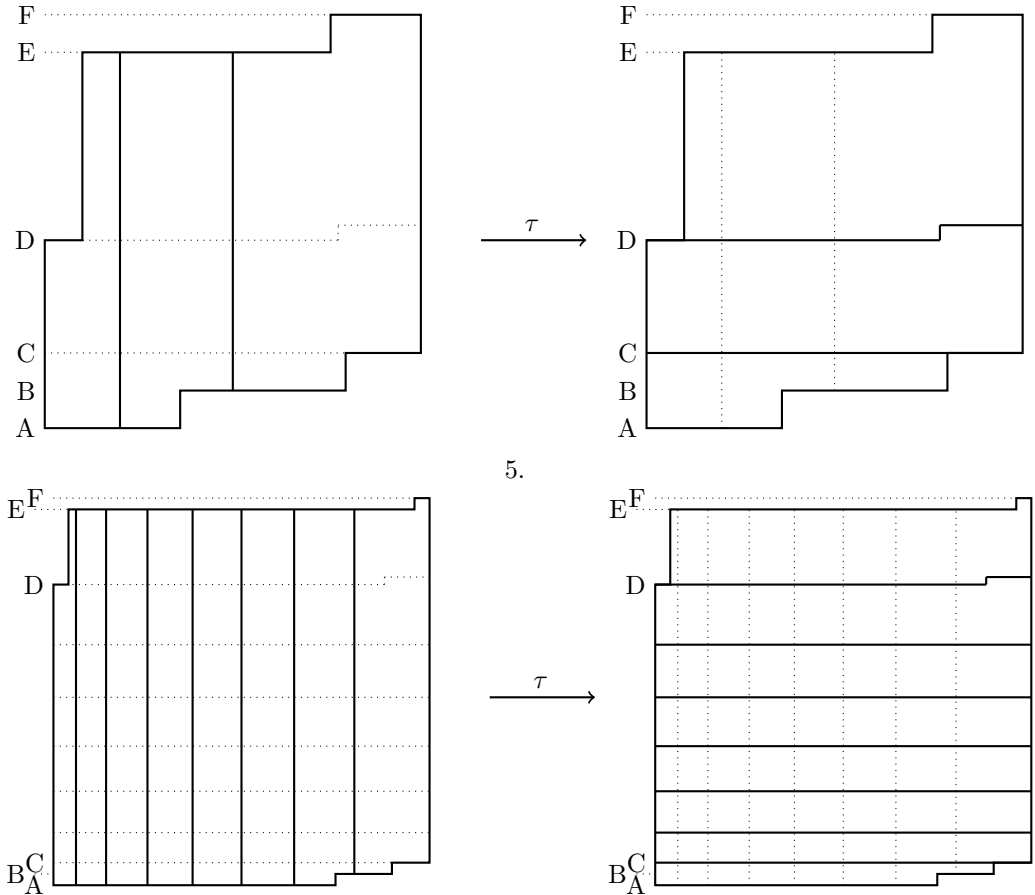
\begin{figure}[htp]
\begin{center}
\begin{tikzpicture}
\draw[black,  thick] (0,0) -- (1.8,0) --(1.8,0.5)-- (4,0.5) --(4,1)-- (5,1)--(5,5.5)--(3.8,5.5)--(3.8,5)--(0.5,5)--(0.5,2.5)--(0,2.5) -- cycle;
\draw[black, thick] (2.5,0.5) -- (2.5,5);
\draw[black, thick] (1,0) -- (1,5);

\filldraw[black] (0,0)   node[anchor=east]{ A};
\filldraw[black] (0,0.5)   node[anchor=east]{ B};
\filldraw[black] (0,1)   node[anchor=east]{ C};
\filldraw[black] (0,2.5)   node[anchor=east]{D };
\filldraw[black] (0,5)   node[anchor=east]{ E};
\filldraw[black] (0,5.5)   node[anchor=east]{ F};

\draw[dotted, black] (0,1) -- (5,1);
\draw[dotted, black] (0,2.5) -- (3.9,2.5);
\draw[dotted, black] (3.9,2.5) -- (3.9,2.7);
\draw[dotted, black] (3.9,2.7) -- (5,2.7);

\draw[ black,thick, ->] (5.8,2.5)-- (7.2,2.5) ;

\draw[dotted, black] (0,5) -- (1,5);
\draw[dotted, black] (0,5.5) -- (5,5.5);

\filldraw[black] (6.7,2.7)   node[anchor=east]{$\tau$ };

\draw[black,  thick] (8+0,0) -- (8+1.8,0) --(8+1.8,0.5)-- (8+4,0.5) --(8+4,1)-- (8+5,1)--(8+5,5.5)--(8+3.8,5.5)--(8+3.8,5)--(8+0.5,5)--(8+0.5,2.5)--(8+0,2.5) -- cycle;
\draw[dotted, black] (8+2.5,0.5) -- (8+2.5,5);
\draw[dotted, black] (8+1,0) -- (8+1,5);

\filldraw[black] (8+0,0)   node[anchor=east]{ A};
\filldraw[black] (8+0,0.5)   node[anchor=east]{ B};
\filldraw[black] (8+0,1)   node[anchor=east]{ C};
\filldraw[black] (8+0,2.5)   node[anchor=east]{D };
\filldraw[black] (8+0,5)   node[anchor=east]{ E};
\filldraw[black] (8+0,5.5)   node[anchor=east]{ F};

\draw[black, thick] (8+0,1) -- (8+5,1);
\draw[ black, thick] (8+0,2.5) -- (8+3.9,2.5);
\draw[black, thick] (8+3.9,2.5) -- (8+3.9,2.7);
\draw[black, thick] (8+3.9,2.7) -- (8+5,2.7);

\draw[dotted, black] (8+0,5) -- (8+1,5);
\draw[dotted, black] (8+0,5.5) -- (8+5,5.5);
\end{tikzpicture}

5.

\begin{tikzpicture}
\draw[black,  thick] (0,0) -- (5-1.25,0) --(5-1.25,0.15)-- (5-0.5,0.15)--(5-0.5,0.3) --(5,0.3)--(5,5.15)--(4.8,5.15)--(4.8,5)--(0.2,5)--(0.2,4)--(0,4) -- cycle;
\draw[black, thick] (4,0.15) -- (4,5);
\draw[black, thick] (3.2,0) -- (3.2,5);
\draw[black, thick] (2.5,0) -- (2.5,5);
\draw[black, thick] (1.85,0) -- (1.85,5);
\draw[black, thick] (1.25,0) -- (1.25,5);
\draw[black, thick] (0.7,0) -- (0.7,5);
\draw[black, thick] (0.3,0) -- (0.3,5);

\filldraw[black] (0,0)   node[anchor=east]{ A};
\filldraw[black] (0,0.3)   node[anchor=east]{ C};
\filldraw[black] (0-0.25,0.15)   node[anchor=east]{B };
\filldraw[black] (0,1.25)   node[anchor=east]{ };
\filldraw[black] (0,1.85)   node[anchor=east]{ };
\filldraw[black] (0,5.15)   node[anchor=east]{F };
\filldraw[black] (0-0.25,5)   node[anchor=east]{E };
\filldraw[black] (0,4)   node[anchor=east]{ D};

\draw[dotted, black] (0,0.3) -- (5,0.3);
\draw[dotted, black] (0,0.7) -- (5,0.7);
\draw[dotted, black] (0,1.25) -- (5,1.25);
\draw[dotted, black] (0,1.85) -- (5,1.85);
\draw[dotted, black] (0,2.5) -- (5,2.5);
\draw[dotted, black] (0,3.2) -- (5,3.2);
\draw[dotted, black] (0,4) -- (4.4,4);
\draw[dotted, black] (4.4,4) -- (4.4,4.1);
\draw[dotted, black] (4.4,4.1) -- (5,4.1);

\draw[ black,thick, ->] (5.8,2.5)-- (7.2,2.5) ;

\draw[dotted, black] (0-0.25,5) -- (0.63,5);
\draw[dotted, black] (0,5.15) -- (5,5.15);

\filldraw[black] (6.7,2.7)   node[anchor=east]{$\tau$ };

\draw[black,  thick] (8+0,0) -- (8+5-1.25,0) --(8+5-1.25,0.15)-- (8+5-0.5,0.15)--(8+5-0.5,0.3) --(8+5,0.3)--(8+5,5.15)--(8+4.8,5.15)--(8+4.8,5)--(8+0.2,5)--(8+0.2,4)--(8+0,4) -- cycle;
\draw[dotted, black] (8+4,0.15) -- (8+4,5);
\draw[dotted, black] (8+3.2,0) -- (8+3.2,5);
\draw[dotted, black] (8+2.5,0) -- (8+2.5,5);
\draw[dotted, black] (8+1.85,0) -- (8+1.85,5);
\draw[dotted, black] (8+1.25,0) -- (8+1.25,5);
\draw[dotted, black] (8+0.7,0) -- (8+0.7,5);
\draw[dotted, black] (8+0.3,0) -- (8+0.3,5);
\draw[dotted, black] (0-0.25,0.15) -- (0,0.15);

\filldraw[black] (8+0,0)   node[anchor=east]{ A};
\filldraw[black] (8+0,0.3)   node[anchor=east]{ C};
\filldraw[black] (8+0-0.25,0.15)   node[anchor=east]{B };
\filldraw[black] (8+0,1.25)   node[anchor=east]{ };
\filldraw[black] (8+0,1.85)   node[anchor=east]{ };
\filldraw[black] (8+0,5.15)   node[anchor=east]{F };
\filldraw[black] (8+0-0.25,5)   node[anchor=east]{E };
\filldraw[black] (8+0,4)   node[anchor=east]{ D};

\draw[black, thick] (8+0,0.3) -- (8+5,0.3);
\draw[black, thick] (8+0,0.7) -- (8+5,0.7);
\draw[black, thick] (8+0,1.25) -- (8+5,1.25);
\draw[black, thick] (8+0,1.85) -- (8+5,1.85);
\draw[black, thick] (8+0,2.5) -- (8+5,2.5);
\draw[black, thick] (8+0,3.2) -- (8+5,3.2);
\draw[black, thick] (8+0,4) -- (8+4.4,4);
\draw[black, thick] (8+4.4,4) -- (8+4.4,4.1);
\draw[black, thick] (8+4.4,4.1) -- (8+5,4.1);

\draw[dotted, black] (8+0-0.25,5) -- (8+0.63,5);
\draw[dotted, black] (8+0,5.15) -- (8+5,5.15);
\draw[dotted, black] (8+0-0.25,0.15) -- (8+0,0.15);
\end{tikzpicture}
\caption{$\Omega_{\alpha}$ and ${\mathcal T}_{\alpha}(\Omega_{\alpha})$ with (a):  $\alpha\in{X_{N,d,i,1}}$, $d=1, i=3$; (b): $\alpha\in{X_{N,d,i,1}}$, $d=2, i=7$ for $N=3$.}
\label{fig:OurNaturalExtensions2}
\end{center}
\end{figure}

\begin{Theorem}\label{thm:NaturalExtension}
Let $N\geq 2$ be an integer, and let $d\geq 1$ and $i\geq 2$ be integers, such that $N=\frac{d(d+i)}{i-1}$. Let $\alpha\in {\overline X}_{N,d,i}$ arbitrary, and let the planar domain $\Omega_{\alpha}$ be the polygon, as given in the statement of Theorem~\ref{thm:PlanarDomainNaturalExtension}, where the values of $A$, $B$, $C$, $D$ and $E$ are also given in Theorem~\ref{thm:PlanarDomainNaturalExtension}.

Consider the probability measure ${\bar \mu}_{\alpha}$ on $\Omega_{\alpha}$, with density $d_{\alpha}$ given by
$$
d_{\alpha}(x,y) = H\cdot \frac{N}{(N+xy)^2}1_{\Omega_{\alpha}}(x,y),
$$
where
\begin{equation}\label{eq:NormalizingConstant}
H^{-1} = 2\log A+2\log(B+1)-\log\big(N-(A+1)d\big)-\log\big(N-(d+i)B\big) ,
\end{equation}
is the normalising constant of $\bar{\mu}_{\alpha}$. Then one easily sees that ${\bar \mu}_{\alpha}$ is ${\mathcal T}_{\alpha}$-invariant. Let $\bar{\mathcal B}_{\alpha}$ be the collection of Borel sets of $\Omega_{\alpha}$. Then the dynamical system $(\Omega_{\alpha}, {\bar{\mathcal B}}_{\alpha}, {\bar \mu}, {\mathcal T}_{\alpha})$ is ergodic. It is also the natural extension of the ergodic system $(I_{\alpha}, {\mathcal B}_{\alpha}, \mu_{\alpha}, T_{\alpha})$, where ${\mathcal B}_{\alpha}$ is the collection of Borel sets of $I_{\alpha}$ and $\mu_{\alpha}$ is the projection of  ${\bar \mu}_{\alpha}$ on the first coordinate (i.e.~on $I_{\alpha}$).

Furthermore, the density $f_{\alpha}(x)$ of the $T_{\alpha}$-invariant measure $\mu_{\alpha}$ is given by
\begin{eqnarray*}\label{eq:onedimensionaldensity}
f_{\alpha}(x) &=& H\Big( \frac{D}{N+Dx}\mathbf{1}_{(\alpha,T(\alpha)+1)}(x)+\frac{E}{N+Ex}\mathbf{1}_{(T(\alpha)+1),T^2(\alpha)}(x)
     +\frac{F}{N+Fx}\mathbf{1}_{(T^2(\alpha),\alpha+1)}(x) \\
     && - \frac{A}{N+Ax}\mathbf{1}_{(\alpha,T^2(\alpha)+1)}(x) - \frac{B}{N+Bx}\mathbf{1}_{(T^2(\alpha)+1),T(\alpha))}(x)
 -\frac{C}{N+Cx}\mathbf{1}_{(T(\alpha),\alpha+1)}(x)\Big) .
\end{eqnarray*}
where $H$ is given by~(\ref{eq:NormalizingConstant}).
\end{Theorem}

\begin{proof}
The proof of this Theorem is nowadays largely routine; that ${\bar \mu}$ is a ${\mathcal T}_{\alpha}$-invariant probability measure is a Jacobian calculation (c.f.~p.~3189 of~\cite{DKW} and pp.~90 and 136 of~\cite{DKa}): for $j\in\{ d,\dots,d+i\}$, let $(x,y)\in\Theta_j$, where $\Theta_j$ is defined as in~(\ref{def:Thetaj}). We already saw that $\Theta_k\cup \Theta_{\ell}=\emptyset$ as $k,\ell\in\{d,\dots,d+i\}$ and $k\neq \ell$. Furthermore, apart from a set of measure zero we have that $\Omega_{\alpha} = \bigcup_{j=d}^{d+i}\Theta_j = \bigcup_{j=d}^{d+i}{\mathcal T}_{\alpha}(\Theta_j)$. Now let $R\subset \Theta_j$, for some $j\in\{ d,\dots,d+i\}$, and set $S={\mathcal T}_{\alpha}(R)$. Setting
$$
u = \frac{N}{x} - j,\,\, v = \frac{N}{j+y},\quad \text{we have that }\,\, x = \frac{N}{u+j},\,\, y = \frac{N}{v} - j,
$$
from which we find that the Jacobian of ${\mathcal T}_{\alpha}$ is given by:
$$
\left| \frac{\partial (x,y)}{\partial (u,v)}\right| = \frac{N^2}{v^2(n+j)^2}.
$$
So we see that $\bar{\mu}$ is ${\mathcal T}_{\alpha}$-invariant, as
$$
\iint_R \frac{N}{(N+xy)^2}\, {\rm d}x{\rm d}y = \iint_S \frac{N}{(N+x(u,v)y(u,v))^2} \left| \frac{\partial (x,y)}{\partial (u,v)}\right| \, {\rm d}u{\rm d}v = \iint_S \frac{N}{(N+uv)^2}\, {\rm d}u{\rm d}v .
$$
In Theorem~\ref{thm:PlanarDomainNaturalExtension} we obtained that ${\mathcal T}_{\alpha}\to\Omega_{\alpha}$ is almost surely bijective, and in Theorem~\ref{ergodic} we already saw that $(I_{\alpha}, {\mathcal B}_{\alpha}, \mu_{\alpha}, T_{\alpha})$ is an ergodic system. That $(\Omega_{\alpha}, {\bar{\mathcal B}}_{\alpha}, {\bar \mu}, {\mathcal T}_{\alpha})$ is the natural extension of $(I_{\alpha}, {\mathcal B}_{\alpha}, \mu_{\alpha}, T_{\alpha})$ can be seen by adapting the proof from~\cite{DKW} or from applying Definition~5.3.1 from~\cite{DKa}. Due to Theorem~\ref{ergodic} and Theorem~5.3.1($iii$) from~\cite{DKa} we now have that the natural extension $(\Omega_{\alpha},\bar{\mathcal B}_{\alpha}, \bar{\mu}_{\alpha}, {\mathcal T}_{\alpha})$ is ergodic. In fact, stronger mixing properties hold, but we do not investigate these here.

Perhaps the most surprising fact is, is that the normalising constant $H$ is \emph{constant} for $\alpha\in {\overline X}_{N,d,i}$. One way to see this is by brute force calculations, as we will do in the rest of this proof. However, in Section~\ref{sec:quilting} we will see that for any two $\alpha,\alpha'\in X_{N,d,i}$ we have that the planar natural extensions of the underlying dynamical systems are metrically isomorphic. This will not only yield that the normalizing constant $H$ is constant for $\alpha\in X_{N,d,i}$, but also that the entropy for all these dynamical systems is equal.

To obtain the normalising constant, we project the density ${\bar \mu}$ of the planar natural extension $\Omega_{\alpha}$ on the first coordinate by integrating out the second coordinate:
\begin{eqnarray*}
H^{-1} &=&\int^{\alpha+1}_{\alpha}\frac{D}{N+Dx}\mathbf{1}_{(\alpha,T(\alpha+1))}
     +\int^{\alpha+1}_{\alpha}\frac{E}{N+Ex}\mathbf{1}_{(T(\alpha+1)),T^2(\alpha))}
     +\int^{\alpha+1}_{\alpha}\frac{F}{N+Fx}\mathbf{1}_{(T^2(\alpha),\alpha+1)}\\
     &&-\int^{\alpha+1}_{\alpha}\frac{A}{N+Ax}\mathbf{1}_{(\alpha,T^2(\alpha+1))}
     -\int^{\alpha+1}_{\alpha}\frac{B}{N+Bx}\mathbf{1}_{(T^2(\alpha+1)),T(\alpha))}
 -\int^{\alpha+1}_{\alpha}\frac{C}{N+Cx}\mathbf{1}_{(T(\alpha),\alpha+1)} \\
 &=&\log  \big(\frac{N+DT(\alpha+1)}{N+D\alpha}\big)+\log  \big(\frac{N+ET^2(\alpha)}{N+ET(\alpha+1)}\big)+\log  \big(\frac{N+F(\alpha+1)}{N+FT^2(\alpha)}\big) \\
  &&-\log  \big(\frac{N+AT^2(\alpha+1))}{N+A\alpha}\big)-\log  \big(\frac{2+BT(\alpha)}{N+BT^2(\alpha+1))}\big)+\log  \big(\frac{N+C(\alpha+1)}{N+C(T(\alpha))}\big) \\
&=&\log  \big(\frac{N+D\frac{-d\alpha+N-d}{\alpha+1}}{N+D\alpha}\big)+\log  \big(\frac{N+E\frac{(d^2 + di + N)\alpha -Nd }{-(d+i)\alpha + N}}{N+E\frac{-d\alpha+N-d}{\alpha+1}}\big)+\log  \big(\frac{N+F(\alpha+1)}{N+F\frac{(d^2 + di + N)\alpha -Nd }{-(d+i)\alpha + N}}\big)\\
  &&-\log  \big(\frac{N+A\frac{(d^2 + di + N)\alpha + d^2 + (-N + i)d - N(i - 1)}{-d\alpha + N - d}}{N+A\alpha}\big)-\log  \big(\frac{N+B\frac{-(d+i)\alpha+N}{\alpha}}{N+B\frac{(d^2 + di + N)\alpha + d^2 + (-N + i)d - N(i - 1)}{-d\alpha + N - d}}\big)\\
  &&-\log  \big(\frac{N+C(\alpha+1)}{N+C\frac{-(d+i)\alpha+N}{\alpha}}\big) \\
\phantom{\hat{H}} &=&\log  \big(\frac{(N-Dd)\alpha+N+D(N-d)}{N+D\alpha}\big)-\log  (\alpha+1)\nonumber\\
  &~&+\log  \big(\frac{\big((E-d-i)N+Ed(d+i)\big)\alpha + N^2- ENd}{(N-Ed)\alpha + E(N-d)+N}\big)+\log  (\alpha+1)-\log  (N-(d + i)\alpha)\\
  &~&+\log  \big(\frac{F\alpha + N+F}{\big(N(F - d - i) + F(d^2 + di)\big)\alpha - NFd + N^2}\big)+\log  (N-(d + i)\alpha) \\
  &~&-\log  \big(\frac{\big(N(-d + A) + Ad(d + i)\big)\alpha + N^2 -\big((A+1)d +A(i - 1)\big)N + Ad(d + i)}{A\alpha+N}\big)+\log  (N-d-d\alpha)\nonumber\\
    &~&-\log \big(\frac{\big(B(d + i)-N\big)\alpha -NB}{(Nd-(N + d(d + i))B)\alpha + ((d + i - 1)N - d(d + i))B - (N-d )N}\big)-\log (N-d-d\alpha)+\log  \alpha \\
  &~&-\log \big(\frac{C\alpha + C + N}{\big(N - (d + i)C\big)\alpha + NC}\big)-\log \alpha \\
 \end{eqnarray*}

Thus, one has
\begin{eqnarray*}
H^{-1} &=& \log  \big(\frac{(N-Dd)\alpha+N+D(N-d)}{N+D\alpha}\big) \\
  &~&+\log  \big(\frac{\big((E-d-i)N+Ed(d+i)\big)\alpha + N^2- ENd}{(N-Ed)\alpha + E(N-d)+N}\big) \\
  &~&+\log  \big(\frac{F\alpha + N+F}{\big(N(F - d - i) + F(d^2 + di)\big)\alpha- NFd + N^2}\big) \\
  &~&-\log  \big(\frac{\big(N(-d + A) + Ad(d + i)\big)\alpha + N^2 -\big((A+1)d +A(i - 1)\big)N + Ad(d + i)}{A\alpha+N}\big) \\
  &~&-\log \big(\frac{\big(B(d + i)-N\big)\alpha -NB}{(Nd-(N + d(d + i))B)\alpha + ((d + i - 1)N - d(d + i))B - (N-d )N}\big) \\
  &~&-\log \big(\frac{C\alpha + C + N}{\big(N - (d + i)C\big)\alpha + NC}\big).
\end{eqnarray*}

We will show now that the first and second terms are constants; the other terms have similar proofs of being constant.

($i$): If we can show that $\frac{N-Dd}{D}=\frac{N+D(N-d)}{N}$, then $\log\big(\frac{(N-Dd)\alpha+N+D(N-d)}{N+D\alpha}\big)$ is a constant.

To see why this last statement holds, assume that $\frac{N-Dd}{D}=\frac{N+D(N-d)}{N}$. Now,
$$
\frac{(N-Dd)\alpha+N+D(N-d)}{N+D\alpha} = \frac{D\cdot \frac{(N-D\alpha )}{D}\alpha + N\cdot \frac{N+D(N-D)}{N}}{N+D\alpha} = \frac{D\alpha +N}{N+D\alpha}\cdot \frac{N-Dd}{D},
$$
and for $\alpha\in X_{N,d,i}$ one has that $\frac{N-Dd}{D}$ is a constant. By substituting $D=\frac{N}{d+1+B}$ (c.f.~(\ref{eq:RelationsBetweenHeights2})), the following statements are equivalent:
\begin{eqnarray*}
& & \frac{N-Dd}{D} = \frac{N+D(N-d)}{N}\quad \Longleftrightarrow \quad  \frac{N}{D}-d = 1+D-\frac{Dd}{N} \\
& \Longleftrightarrow & \frac{N}{\frac{N}{d+1+B}}-d = 1+\frac{N}{d+1+B}-\frac{\frac{N}{d+1+B}d}{N} \quad \Longleftrightarrow  \quad
B = \frac{N-d}{d+1+B}.
\end{eqnarray*}
Meanwhile, from~(\ref{eq:ValuesOfTheVariusHeights1}) we know that $B=\frac{N}{B+1}-d$, so $(B+d)(B+1)=N$, and from this we see that  $B(d+1+B)=N-d$; i.e.\ we find that $B=\frac{N-d}{d+1+B}$. Therefore, $\frac{N-Dd}{D}=\frac{N+D(N-d)}{N}$, and $\log\big(\frac{(N-Dd)\alpha+N+D(N-d)}{N+D\alpha}\big)$ is a constant. Note that $\frac{N-Dd}{D}=\frac{N}{D}-d=1+B$.

($ii$): The proof that the second term is constant is similar to case ($i$). Here we show that \emph{if} $\frac{(E-d-i)N+Ed(d+i)}{N-Ed}=\frac{ N^2- ENd}{E(N-d)+N}$, \emph{then} $\log\big(\frac{((E-d-i)N+Ed(d+i))\alpha + N^2- ENd}{(N-Ed)\alpha + (E+1)N- Ed}\big)$ is a constant. To see this last statement, note that \emph{if} we assume that $\frac{(E-d-i)N+Ed(d+i)}{N-Ed}=\frac{ N^2- ENd}{E(N-d)+N}$, we have that
$\frac{((E-d-i)N+Ed(d+i))\alpha + N^2- ENd}{(N-Ed)\alpha + (E+1)N- Ed}$ is equal to
\begin{eqnarray*}
&& \frac{(N-Ed)\cdot \frac{(E-d-i)N+Ed(d+i))}{N-Ed}\cdot \alpha + \frac{N^2- ENd}{(E(N-d)+N}\cdot (E(N-d)+N)}{(N-Ed)\alpha + (E+1)N- Ed} \\
&=& \frac{(N-Ed)\alpha + E(N-d) + N}{(N-ED)\alpha + (E+1)N-Ed}\cdot \frac{N^2-ENd}{E(N-d) +N},
\end{eqnarray*}
and for $\alpha\in X_{N,d,i}$ we have that $\frac{N^2-ENd}{E(N-d) +N}$ is a constant. In order to see that we indeed have $\frac{(E-d-i)N+Ed(d+i)}{N-Ed}=\frac{ N^2- ENd}{E(N-d)+N}$, note that from~(\ref{eq:A2}) we know that $\frac{N}{A+1}-d=\frac{N}{A+d+i}$, and therefore that we have that $d(A+1)(d+i)=N(d+i)-N-dA(A+1)$.

From this last expression, and since $E=A+1$,
 \begin{eqnarray*}
\frac{(E-d-i)N+Ed(d+i)}{N-Ed}&=&\frac{(1+A-d-i)N+d(1+A)(d+i)}{N-(1+A)d} \\
&=& \frac{(1+A-d-i)N+N(d+i)-N-dA(A+1)}{N-(A+1)d} \\
&=& \frac{NA-dA(A+1)}{N-(A+1)d}) \, =\, A
\end{eqnarray*}

Now from~(\ref{eq:A1}) it immediately follows that $\frac{N}{d+i+A+1}=A$, and once more using that $\frac{N}{A+1}-d=\frac{N}{A+d+i}$, we find that:
$$
\frac{N^2- ENd}{E(N-d)+N} = \frac{N(\frac{N}{E}-d)}{N-d+\frac{N}{E}} = \frac{N(\frac{N}{A+1}-d)}{N+\frac{N}{A+1}-d} = N\cdot\frac{\frac{N}{A+d+i}}{N+\frac{N}{A+d+i}} = \frac{N}{d+i+A+1} = A.
$$
Therefore, $\frac{(E-d-i)N+Ed(d+i)}{N-Ed}=\frac{ N^2- ENd}{E(N-d)+N}$; the second term $\log\big(\frac{((E-d-i)N+Ed(d+i))\alpha + N^2- ENd}{(N-Ed)\alpha + (E+1)N- Ed}\big)$ is a constant.

Similarly, one can show that the other terms in $H^{-1}$ are also constant. In doing so one finds that:
$$
H^{-1} = \log(1+B)+\log A + \log\left( \frac{1}{N-(d+i)B}\right) - \log\big(N-(A+1)d\big) - \log\left( \frac{1}{B+1}\right) - \log\left( \frac{1}{A}\right) ,
$$
which can be simplified into ~(\ref{eq:NormalizingConstant}):
$$
H^{-1} = 2\log A+2\log(B+1)-\log\big(N-(A+1)d\big)-\log\big(N-(d+i)B\big) .
$$
\end{proof}

\section{Quilting}\label{sec:quilting}
In this section we will use a technique called \emph{quilting}, which was first used in~\cite{KSSm}, and which was originally based on the use of `insertions' and `singularizations' (which are operations on the partial quotients of the regular continued fraction expansion of any real number $x$) as investigated in~\cite{DK,dJKN,K}. In these last three paper quilting was steered by the insertion and singularization operations, while quilting was `blind' in~\cite{KL,KSSm}, in the sense that the quilting maps were `guessed'. For $N$-expansions with finitely many digits quilting was already used in~\cite{KL} for the case $N=2$, in order to show the occurrence of an `entropy plateau'; i.e.\ a set of values of $\alpha$ for which the dynamical systems $(I_{\alpha},{\mathcal B}_{\alpha}, \mu_{\alpha}, T_{\alpha})$ all have the same entropy. In this section we will see that the approach from~\cite{KL} can be generalized to the intervals $(A,B)$ we obtained in Theorem~\ref{thm:IntervalOfEqualEntropy}, where  $A$ and $B$ are given in~(\ref{eq:ValuesOfTheVariusHeights1}). Recently, quilting has received a thorough theoretic founding in~\cite{CKS}, but here we follow the more `hands-on' approach from~\cite{KL,KSSm}.

Let $N\geq 2$ be an integer, and let $d\geq 1$ and $i\geq 2$ be integers, such that $N=\frac{d(d+i)}{i-1}$. Let $\alpha, \beta\in X_{N,d,i,k}$, $\alpha < \beta$ arbitrary, with $k\in \{ d,\dots,d+i-1\}$, and let the planar domains $\Omega_{\alpha}$ and $\Omega_{\beta}$ be the polygon, as given in the statement of Theorem~\ref{thm:PlanarDomainNaturalExtension}, where the values of $A$, $B$, $C$, $D$ and $E$ are also given in Theorem~\ref{thm:PlanarDomainNaturalExtension}. Since $\alpha < \beta$ both maps $T_{\alpha}:I_{\alpha}\to I_{\alpha}^{-}$ and $T_{\beta}:I_{\beta}\to I_{\beta}^{-}$ have cylinders for the digits $d,\dots,d+i$, but that these cylinder are no the same (but overlapping) for the same digit $k$ (with $k\in\{ d,\dots,d+i\}$). For this reason we define the cylinders of $T_{\alpha}$ by $I_k(\alpha)$, and similarly those for the map $T_{\beta}$ by $I_k(\beta)$. We first assume that $T_{\alpha}^{2}(\alpha)\in I_k^o(\alpha), T^2_{\alpha}(\alpha+1)\in I_{k+1}^o(\alpha)$, and similarly $T_{\beta}^{2}(\beta)\in I_k^o(\beta), T^2_{\beta}(\beta+1)\in I_{k+1}^o(\beta)$; c.f.~(\ref{def:X(N,d,i,k)}).

As in~\cite{KL}, we define sets $A_0$, $A_1={\mathcal T}_{\alpha}(A_0)$ and $A_2={\mathcal T}_{\alpha}(A_1)$, and sets $D_0$, $D_1={\mathcal T}_{\beta}(D_0)$ and $D_2={\mathcal T}_{\beta}(D_1)$, where:
\begin{equation}\label{def:definitionsOfA1AndD1}
A_0=[\alpha,\beta]\times [A,D],\quad D_0=[\alpha +1,\beta+1]\times [C,F].
\end{equation}
We have the following result.
\begin{Proposition}\label{prop:LemmaAboutAandD}
Let $N\geq 2$ be an integer, and let $d\geq 1$ and $i\geq 2$ be integers, such that $N=\frac{d(d+i)}{i-1}$. Let $\alpha, \beta\in X_{N,d,i,k}$, $\alpha < \beta$ arbitrary, with $k\in \{ d,\dots,d+i-1\}$. Moreover, let $T_{\alpha}^{2}(\alpha)\in I_k^o(\alpha), T^2_{\alpha}(\alpha+1)\in I_{k+1}^o(\alpha)$, and similarly $T_{\beta}^{2}(\beta)\in I_k^o(\beta), T^2_{\beta}(\beta+1)\in I_{k+1}^o(\beta)$.

Then we have that:
\begin{equation}\label{eq:QuiltingWorks}
{\mathcal T}_{\alpha}(A_2) = {\mathcal T}_{\beta}(D_2),
\end{equation}
and the map ${\mathcal M}_{\beta , \alpha}: \Omega_{\beta}\to \Omega_{\alpha}$, defined by:
\begin{equation}\label{eq:MetricIsomorphism}
{\mathcal M}_{\beta ,\alpha}(x,y) = \begin{cases}
(x,y), & \text{if $(x,y)\in \Omega_{\beta}\setminus (D_0\cup D_1\cup D_2)$}; \\
{\mathcal T}_{\alpha}^{-3}\Big( {\mathcal T}_{\beta}^3(x,y)\Big), & \text{if $(x,y)\in D_0$}; \\
{\mathcal T}_{\alpha}^{-2}\Big( {\mathcal T}_{\beta}^2(x,y)\Big), & \text{if $(x,y)\in D_1$}; \\
{\mathcal T}_{\alpha}^{-1}\Big( {\mathcal T}_{\beta}(x,y)\Big), & \text{if $(x,y)\in D_2$}.
\end{cases}
\end{equation}
is a metric isomorphism from $\Omega_{\beta}$ to $\Omega_{\alpha}$; cf.~Figure~\ref{fig:OmegaAlphaAndBeta}.
\end{Proposition}

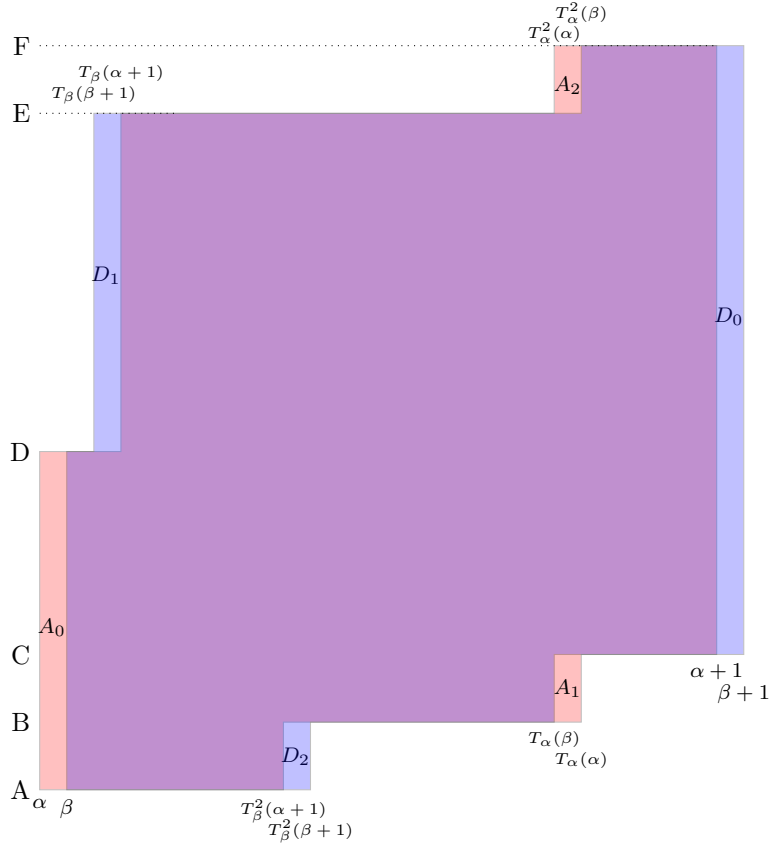
\begin{figure}[h!]
\begin{center}
\begin{tikzpicture}[scale=1.8]

\draw[fill=red,opacity=0.25] (0,0) -- (1.8,0)
--(1.8,0.5)-- (4,0.5) --(4,1)-- (5,1)--(5,5.5)--(3.8,5.5)--(3.8,5)--(0.6,5)--(0.6,2.5)--(0,2.5) -- cycle;
\filldraw[black] (0,0)   node[anchor=east]{ A};
\filldraw[black] (0,0.5)   node[anchor=east]{ B};
\filldraw[black] (0,1)   node[anchor=east]{ C};
\filldraw[black] (0,2.5)   node[anchor=east]{D };
\filldraw[black] (0,5)   node[anchor=east]{ E};
\filldraw[black] (0,5.5)   node[anchor=east]{ F};


\node at (0.09,1.2){\color{black}\footnotesize $A_0$};
\node at (3.9,0.75){\color{black}\footnotesize $A_1$};
\node at (3.9,5.21){\color{black}\footnotesize $A_2$};
\node at (1.89,0.25){\color{black}\footnotesize $D_2$};
\node at (0.49,3.8){\color{black}\footnotesize $D_1$};
\node at (5.09,3.5){\color{black}\footnotesize $D_0$};

\node[anchor=north] at (0,0){\color{black}\footnotesize $\alpha$};
\node[anchor=north] at (0.2,0){\color{black}\footnotesize $\beta$};

\node[anchor=north] at (5,1){\color{black}\footnotesize $\alpha+1$};
\node[anchor=north] at (5.2,0.85){\color{black}\footnotesize $\beta+1$};
\node[anchor=north] at (3.8,0.5){\color{black}\tiny $T_{\alpha}(\beta)$};
\node[anchor=north] at (4,0.35){\color{black}\tiny $T_{\alpha}(\alpha)$};
\node[anchor=south] at (4,5.6){\color{black}\tiny $T_{\alpha}^2(\beta)$};
\node[anchor=south] at (3.8,5.45){\color{black}\tiny $T_{\alpha}^2(\alpha)$};
\node[anchor=south] at (0.4,5){\color{black}\tiny $T_{\beta}(\beta+1)$};
\node[anchor=south] at (0.6,5.15){\color{black}\tiny $T_{\beta}(\alpha+1)$};
\node[anchor=north] at (2,-0.15){\color{black}\tiny $T_{\beta}^2(\beta+1)$};
\node[anchor=north] at (1.8,0){\color{black}\tiny $T_{\beta}^2(\alpha+1)$};

\draw[fill=blue,opacity=0.25] (0.2,0) -- (1.8+0.2,0)
--(1.8+0.2,0.5)-- (4-0.2,0.5) --(4-0.2,1)-- (5.2,1)--(5.2,5.5)--(3.8+0.2,5.5)--(3.8+0.2,5)--(0.4,5)--(0.4,2.5)--(0.2,2.5) -- cycle;

\draw[dotted, black] (0,5) -- (1,5);
\draw[dotted, black] (0,5.5) -- (5,5.5);
\end{tikzpicture}
\end{center}
\caption{$\Omega_{\alpha}$ and $\Omega_{\beta}$ for $\alpha, \beta\in X_{N,d,i,k}$, where $\alpha < \beta$.}\label{fig:OmegaAlphaAndBeta}
\end{figure}

\begin{proof}
Note that from Theorem~\ref{thm:PlanarDomainNaturalExtension} it follows that $A_1={\mathcal T}_{\alpha}(A_0) = [T_{\beta}(\beta),T_{\alpha}(\alpha)]\times [B,C]$, so due to the construction of the planar natural extension $\Omega_{\alpha}$ we \emph{must} have that $T_{\alpha}(\beta) = T_{\beta}(\beta)$; i.e.\ that $\beta\in\Delta_{d+i}(\alpha )$ (if $\beta\not\in\Delta_{d+i}(\alpha )$ then the ``$\alpha$-digit'' of $\beta$ is at most $d+i-1$, and we would have that $T_{\alpha}(\beta) \geq T_{\beta}(\beta) +1\not\in\Omega_{\alpha}$. In a similar way we find that
$$
A_2={\mathcal T}_{\alpha}^2(A_0) = [T_{\alpha}^2(\alpha), T_{\beta}^2(\beta)]\times [E,F],
$$
that
$$
D_1 = {\mathcal T}_{\beta}(D_0) = [T_{\beta}(\beta +1), T_{\alpha}(\alpha +1)]\times [D,E],\quad
D_2 = {\mathcal T}_{\beta}(D_0) = [T_{\alpha}^2(\alpha +1) , T_{\beta}^2(\beta +1)]\times [A,B].
$$
From this we see that
\begin{equation}\label{eq:T{alpha}(A4)}
{\mathcal T}_{\alpha}(A_2) = [T_{\beta}^3(\beta) , T_{\alpha}^3(\alpha)]\times \left[ \frac{N}{k+F}, \frac{N}{k+E}\right] ,
\end{equation}
and that
\begin{equation}\label{eq:T{beta}(D4)}
{\mathcal T}_{\beta}(D_2) = [T_{\beta}^3(\beta +1) , T_{\alpha}^3(\alpha+1)]\times \left[ \frac{N}{k+1+B}, \frac{N}{k+1+A}\right] .
\end{equation}
Since both $\alpha , \beta\in X_{N,d,i,k}$, it follows from Theorem~\ref{thm:MatchingIn3Steps} that $T_{\alpha}^3(\alpha) = T_{\alpha}^3(\alpha +1)$ and that $T_{\beta}^3(\beta) = T_{\beta}^3(\beta +1)$. So from~(\ref{eq:T{alpha}(A4)}), (\ref{eq:T{beta}(D4)}) and the fact that in~(\ref{eq:ValuesOfTheVariusHeights1}) we saw that $A+1=E$ and $B+1=F$, we immediately find that ${\mathcal T}_{\alpha}^3(A_0) = {\mathcal T}_{\beta}^3(D_0)$.

To see that ${\mathcal M}:\Omega_{\beta}\to\Omega_{\alpha}$ is a metric isomorphism, note that the sets $A_0$, $A_1$ and $A_2$ are disjoint a.s.\  from $\Omega_{\beta}$, and that the sets $D_0$, $D_1$ and $D_2$ are disjoint a.s.\ from $\Omega_{\alpha}$. Furthermore, all four maps ${\mathcal T}_{\alpha}$, ${\mathcal T}_{\alpha}^{-1}$, ${\mathcal T}_{\beta}$ and ${\mathcal T}_{\beta}^{-1}$ preserve any measure with density
$\frac{1}{H} \frac{N}{(N+xy)^2}$, where $H$ is the normalizing constant given in~(\ref{eq:NormalizingConstant}), and the maps  ${\mathcal T}_{\alpha}$, ${\mathcal T}_{\alpha}^{-1}$ are a.s.\ bijective on $\Omega_{\alpha}$, while the maps ${\mathcal T}_{\beta}$ and ${\mathcal T}_{\beta}^{-1}$ are a.s.\ bijective on $\Omega_{\beta}$.

Thus we see that for $\alpha,\beta\in X_{N,d,i,k}$, $\alpha < \beta$, for $k\in\{d,d+1,\dots,d+i\}$, the ergodic dynamical systems $(\Omega_{\alpha},{\mathcal B}_{\alpha},\bar{\mu}_{\alpha}, {\mathcal T}_{\alpha})$ and $(\Omega_{\beta},{\mathcal B}_{\beta},\bar{\mu}_{\beta}, {\mathcal T}_{\beta})$ are metrically isomorphic.
\end{proof}

\begin{Remarks}\label{rem:isomorphisms}{\rm
($i$) To see that for any $\alpha,\beta\in {\bar X}_{N,d,i}$ (say with $\alpha < \beta$) we have that their corresponding dynamical systems are metrically isomorphic it is enough to show that for $\alpha,\beta \in {\bar X}_{N,d,i,k}$ (for some $k\in\{ d,d+1,\dots,d+i\}$) with $\alpha < \beta$, where either $T_{\alpha}(\alpha )= \frac{N}{\alpha + 1+d}$, or $T_{\alpha}(\alpha +1)=\alpha$, or $T_{\alpha}^{2}(\alpha)$ is a boundary point of $I_k(\alpha)$, or $T^2_{\alpha}(\alpha+1)$ is a boundary point of $I_{k+1}(\alpha)$, and similarly for $T_{\beta}$, that the corresponding dynamical systems are isomorphic. In~Remarks~\ref{rem:RemarkAboutTheMatchingTheorem}($iv$) we mentioned as special cases when either $T_{\alpha}(\alpha )= \frac{N}{\alpha + 1+d}$, or $T_{\alpha}(\alpha +1)=\alpha$. We will show in these last two cases that it is very easy to construct a natural extension. All other cases mentioned here are similar to these two cases, and therefore omitted.

Let us first assume that $T_{\alpha}(\alpha )= \frac{N}{\alpha + 1+d}$. As we already remarked in~Remarks~\ref{rem:RemarkAboutTheMatchingTheorem}($iv$), $\frac{N}{\alpha + 1+d}$ is the dividing point between the cylinders $\Delta_d$ and $\Delta_{d+1}$. But then we have, that $T_{\alpha}(\alpha ) = \frac{N}{\alpha+d+1}$, and from~(\ref{eq:BasCF}) we immediately find that $\alpha = B$, and due to $F=B+1$ (cf.~(\ref{eq:ValuesOfTheVariusHeights1})) and $F=\frac{N}{d+B}$ (cf.~(\ref{eq:ValuesOfTheVariusHeights2})), from which we see that $B=\frac{N}{F}- d$ (i.e.~$T_B(B+1)=B$), we see that $T_B$ is full on the right-most cylinder $\Delta_B(d)$. Consequently, the domain of the planar natural extension for this particular value of $\alpha$ is the left-hand side polygon given in Figure~\ref{fig:AandB}.

Next, let us assume that $T_{\alpha}(\alpha+1)=\frac{N}{\alpha + d+i}$. So the map $T_{\alpha}$ sends $\alpha+1$ to the dividing point of the most left-hand cylinder $\Delta_{d+i}$ and $\Delta_{d+i-1}$. From $T_{\alpha}(\alpha+1)=\frac{N}{\alpha + d+i}$ it follows that $\alpha$ is the positive root of $d\alpha^2 + d(d+i+1)\alpha + d(d+i) + (1-(d+i)))N = 0$, which is
\begin{equation}\label{eq:OtherExpressionForA}
\alpha = \frac{-d(d+i+1) + \sqrt{d^2(d+i+1)^2-4d((d+i)+(1-(d+i))N)}}{2d}.
\end{equation}
Now a trivial but somewhat tedious calculation shows that the expression for $A$ from~(\ref{eq:ValuesOfTheVariusHeights1}) and for $\alpha$ from~(\ref{eq:OtherExpressionForA}) are the same whenever $N=\frac{(d+i)}{i-1}$; i.e.\ we find that $T_A(A+1)=\frac{N}{A+ d+i}$. In this case we have that $T_A(A)=A+1$ if and only if $\frac{N}{A}-(d+i)=A+1$, which is equivalent to $A$ satisfying $A^2+(d+i+1)A-N=0$. The positive solution to this equation is $A$ from~(\ref{eq:ValuesOfTheVariusHeights1}). We see that $T_A$ is full on the left-most cylinder $\Delta_A(d+i)$. Consequently, the domain of the planar natural extension for this particular value of $\alpha$ is the right-hand side polygon given in Figure~\ref{fig:AandB}.

($ii$) The map ${\mathcal M}_{\beta,\alpha}$ from~(\ref{eq:MetricIsomorphism}) is the so-called \emph{quilting map} from $\Omega_{\beta}$ to $\Omega_{\alpha}$. Note that ${\mathcal M}_{\beta,\alpha}$ maps $D_{\ell}$ bijectively (a.s.) to $A_{\ell}$ for $\ell = 0,1,2$. That we \emph{quilt} can be seen as follows: we map $D_0$ to $A_0$, so we add $A_0$ to $\Omega_{\beta}$ as a `patch'. We then remove $D_0$, $D_1$ and $D_2$ from $\Omega_{\beta}$, and add to $\Omega_{\beta}$ the sets $A_1$ and $A_2$, thus finding $\Omega_{\alpha}$. Note that we also should remove ${\mathcal T}_{\beta}(D_2)$ from $\Omega_{\beta}$, but the `gap' thus created is filled by ${\mathcal T}_{\alpha}(A_2)$, and the \emph{quilting process} stops.

($iii$) In~\cite{HKLM,dJKN}, the quilting is steered by operations on the partial quotients called insertions and singularizations, and in these cases the map ${\mathcal M}_{\beta,\alpha}$ can be more explicitly given. In fact one only needs to know ${\mathcal M}_{\beta,\alpha}:D_0\to A_0$, remove all forward images $D_{\ell}={\mathcal T}_{\beta}^{\ell}(D_0)$ from $\Omega_{\beta}$, and add all forward images ${\mathcal T}_{\alpha}^{\ell}(A_0)$ to  $\Omega_{\beta}$ (where $\ell = 0,1,2,\dots$), yielding $\Omega_{\alpha}$. One easily sees that the first coordinate map of ${\mathcal M}_{\beta,\alpha}:D_0\to A_0$ must be $x\mapsto x-1$, but the second coordinate map is usually more complicated.}\hfill $\triangle$
\end{Remarks}

From Proposition~\ref{prop:LemmaAboutAandD} and Remarks~\ref{rem:isomorphisms}($i$) we have the following result.

\begin{Theorem}\label{theorem:LemmaAboutAandD}
Let $N\geq 2$ be an integer, and let $d\geq 1$ and $i\geq 2$ be integers, such that $N=\frac{d(d+i)}{i-1}$. Let $\alpha, \beta\in [A,B]= {\overline X}_{N,d,i}$, $\alpha < \beta$ arbitrary. Then the dynamical systems $(\Omega_{\alpha}, \bar{\mathcal B}_{\alpha}, \bar{\mu}_{\alpha}, {\mathcal T}_{\alpha})$ and $(\Omega_{\beta}, \bar{\mathcal B}_{\beta}, \bar{\mu}_{\beta}, {\mathcal T}_{\beta})$ are metrically isomorphic.
\end{Theorem}

\begin{figure}[h!]
\begin{center}
\begin{tikzpicture}
\draw[black,  thick] (0,0) -- (3.3,0) --(3.3,1)-- (5,1)-- (5,1)--(5,5)--(0,5)--(0,0) -- cycle;

\filldraw[black] (0,0)   node[anchor=east]{ A};
\filldraw[black] (0,0.3)   node[anchor=east]{ B};
\filldraw[black] (0,1)   node[anchor=east]{ C};
\filldraw[black] (3.3,0)   circle (1pt)  node[anchor=north]{\tiny $T(\alpha)=\frac{N}{\alpha+d +1}$ };
\filldraw[black] (0,3.5)   node[anchor=east]{ D};
\filldraw[black] (0,5)   node[anchor=east]{ E};
\filldraw[black] (0,5.3)  circle (1pt)   node[anchor=east] { F};
\draw[dotted, black] (3.3,1)-- (3.3,5);

\draw[dotted, black] (0.7+7,3.5)-- (0.7+7,0.3);

\filldraw[black] (4.5, 2.5)   node[anchor=east]{ $\bigtriangleup_{d}$};
\filldraw[black] (3,2.5)   node[anchor=east]{ $\bigtriangleup_{d+i},\cdots,\bigtriangleup_{d+1}$};

\filldraw[black] (0.79+7, 2.5)   node[anchor=east]{\tiny$\bigtriangleup_{d+i}$};
\filldraw[black] (0.79+10.3, 2.5) node[anchor=east]{$\bigtriangleup_{d+i-1},\cdots,\bigtriangleup_{d}$};

\draw[black,  thick] (0+7,0.3) -- (5+7,0.3) --(5+7,5.3)--(0.7+7,5.3)--(0.7+7,3.5)--(0+7,3.5)-- cycle;

\filldraw[black] (0+7,0) circle (1pt)  node[anchor=east]{ A};
\filldraw[black] (0+7,0.3)   node[anchor=east]{ B};
\filldraw[black] (0+7,1)   node[,anchor=east]{ C};
\filldraw[black] (0.7+7,5.3)   circle (1pt)  node[anchor=south]{\tiny $T(\alpha+1)=\frac{N}{\alpha+d +i} $ };
\filldraw[black] (0+7,3.5)   node[anchor=east]{ D};
\filldraw[black] (0+7,5)   node[anchor=east]{ E};
\filldraw[black] (0+7,5.3)   node[anchor=east]{ F};

\end{tikzpicture}
\end{center}
\caption{$\Omega_{B}$ (left) and $\Omega_{A}$ (right).}\label{fig:AandB}
\end{figure}
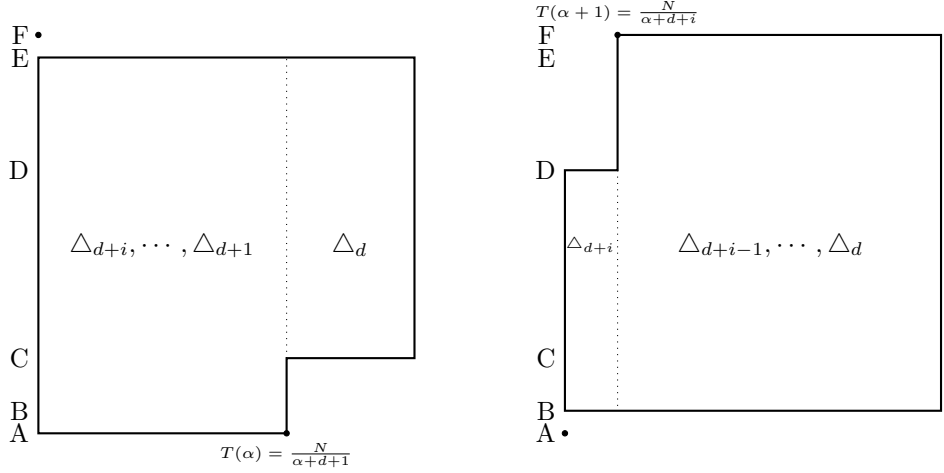

\section{Plateaux with the same entropy for every $N\in\N$, $N\geq 2$}\label{sec:entropy}
Let $N\geq 2$ be an integer, and let $d\geq 1$ and $i\geq 2$ be integers, such that $N=\frac{d(d+i)}{i-1}$. Let $\alpha\in X_{N,d,i}$, then a direct corollary of Theorem~\ref{theorem:LemmaAboutAandD} is that not only the normalizing constat $H$ is the same for all $\alpha$, but also that that the entropy is $h(T_{\alpha})$ is constant for all $\alpha\in X_{N,d,i}$. This is exactly the statement of Theorem~\ref{thm:EqualEntropyOnX(N,d,i)}. In the statement (and proof) of Theorem~\ref{thm:EqualEntropyOnX(N,d,i)} the dilogarithm function ${\rm Li}_2$ appears at various places, and therefore we first recall some facts about the dilogarithm, which for $z\in\C$ can be defined by the sum
\begin{equation}\label{def:dilogarithm}
{\rm Li}_2(z) = \sum_{k=1}^{\infty} \frac{z^k}{k^2},\,\, \text{for $|z| \leq 1$},\qquad \text{or by the integral}\quad {\rm Li}_{2}(z)=\int^{0}_{z}\frac{\log(1-t)}{t}\, {\rm d}t;
\end{equation}
see also~\cite{Le1,Le2} for more information on the dilogarithm function (and polylogarithm functions in general).

\begin{Lemma}\label{lem:dilogarithm}   For any $m>n>0$, and $M,N>0$,
\begin{eqnarray*}\label{eq:dilogaritm1}
\int_{n}^{m} (\log x)\frac{M}{N+Mx}\, {\rm d}x = \Big( {\rm Li}_{2}(-\tfrac{M}{N}x) + (\log x)\log(1+\tfrac{M}{N}x)\Big){\Big |}_{n}^{m},
\end{eqnarray*}
where ${\rm Li}_{2}(.)$ is the dilogarithm function from~(\ref{def:dilogarithm}).
\end{Lemma}

\begin{proof} Note that integration by parts yields that,
\begin{eqnarray*}\label{eq:dilogaritm2}
\int (\log x)\frac{M}{N+Mx}\, {\rm d}x  &=& (\log x)\log(\tfrac{N}{M}+x) - \int\frac{\log(\frac{N}{M}+x)}{x}\, {\rm d}x \\
&=& (\log x)\log(\tfrac{N}{M}+x) - \bigg( \int\frac{\log(1+\frac{M}{N}x)}{x}\, {\rm d}x + \int\frac{\log\frac{N}{M}}{x}\, {\rm d}x\bigg) \\
&=& -\int\frac{\log(1+\frac{M}{N}x)}{x}\, {\rm d}x + (\log x)\log(1+\tfrac{M}{N}x).
\end{eqnarray*}
Setting $\frac{M}{N}x=-t$ one easily sees that, for $m>0$,
\begin{eqnarray*}\label{eq:dilogaritm3}
-\int^{m}_{0}\frac{\log(1+\frac{M}{N}x)}{x}\, {\rm d}x = \int^{0}_{-\tfrac{M}{N}m}\frac{\log(1-t)}{t}\, {\rm d}t = {\rm Li}_{2}(-\tfrac{M}{N}m).
\end{eqnarray*}
Therefore,
\begin{eqnarray*}\label{eq:dilogaritm4}
\int_{n}^{m} (\log x)\frac{M}{N+Mx}\, {\rm d}x = \Big( {\rm Li}_{2}(-\tfrac{M}{N}x) + (\log x)\log(1+\tfrac{M}{N}x)\Big){\Big |}_{n}^{m}.
\end{eqnarray*}
\end{proof}

We have the following result.

\begin{Theorem}\label{thm:EqualEntropyOnX(N,d,i)}
Let $N\geq 2$ be an integer, and let $d, i\in\N$, $i\geq 2$, be such, that $N=\frac{d(d+i)}{i-1}$. Then for any $\alpha\in [A,B] = {\overline X}_{N,d,i}$, one has that the entropy function $h(T_{\alpha})$ is constant on $[A,B] = {\overline X}_{N,d,i}$, and is given by:
\begin{eqnarray*}
h(T_{\alpha}) &=& \log N - 2H\bigg(\Big( {\rm Li}_{2}(-\tfrac{Ex}{N}) + (\log x )\log(\tfrac{Ex}{N}+1)\Big){\Big |}_{B}^{B+1} \\
&& - \Big( {\rm Li}_{2}(-\tfrac{Ax}{N}) + (\log x )\log(\tfrac{Ax}{N}+1)\Big){\Big |}_{B}^{D} - \Big( {\rm Li}_{2}(-\tfrac{Cx}{N}) + (\log x )\log(\tfrac{Cx}{N}+1))\Big)\Big{|}_{D}^{B+1}\bigg),
\end{eqnarray*}
where $H^{-1}= 2\log A+\log (A+1)+\log (B+1)-\log\big( N-(A+1)d\big) - \log\big( N-(d+i)B\big)$ is the normalising constant for the $T_{\alpha}$-invariant measure $\mu_{\alpha}$ for $\alpha \in X_{N,d,i}$.
\end{Theorem}

\begin{proof} From Theorem~\ref{theorem:LemmaAboutAandD} we know for any $\alpha , \beta\in [A,B] = {\overline X}_{N,d,i}$ we have that the dynamical systems $(\Omega_{\alpha},{\mathcal B}_{\alpha},\bar{\mu}_{\alpha}, {\mathcal T}_{\alpha})$ and $(\Omega_{\beta},{\mathcal B}_{\beta},\bar{\mu}_{\beta}, {\mathcal T}_{\beta})$ are metrically isomorphic, and due to Theorem~9.22 from~\cite{DKa} all $\alpha , \beta\in [A,B]$ have the same entropy.

So to know the entropy $h(T_{\alpha})$ it suffices to calculate it for just one $\alpha\in [A,B]$; we choose $\alpha = B$, as in this case (or in the case $\alpha = A$) the shape of $\Omega_{\alpha}$ is the easiest (cf.~Figures~\ref{Fig:NatExtN=2}, \ref{fig:OurNaturalExtensions1}, \ref{fig:OurNaturalExtensions2}, and~\ref{fig:AandB}). In this case we also have that $T_B(B)=\frac{N}{B+d+1}=D$; cf.~(\ref{eq:RelationsBetweenHeights2}).

By using Rohlin's formula (see~\cite{DKa}), and with $T_B$-invariant density $f_B$ from Theorem~\ref{thm:NaturalExtension}, we have for $\alpha \in [A,B]$, and in particular for $\alpha = B$ that:
\begin{eqnarray*}
h(T_{\alpha}) &=& \int_{B}^{B+1} \log|T_B'(x)|\, {\rm d}\mu_B(x) \,\, =\,\, \int_{B}^{B+1} \log|T_B'(x)| f_B(x)\, {\rm d}x \\
&=& \int_{B}^{B+1} (\log N-2\log x)f_B(x)\, {\rm d}x\,\, =\,\, \log N-2\cdot\int_{B}^{B+1} (\log x)\, f_B(x)\, {\rm d}x \\
&=& \log N-2H\cdot\int_{B}^{B+1} (\log x)\, \Big((\frac{E}{N+Ex}-\frac{A}{N+Ax})\mathbf{1}_{(B,D)}(x)\, {\rm d}x \\
&& + (\frac{E}{N+Ex}-\frac{C}{N+Cx})\mathbf{1}_{(D,B+1)(x)}\, {\rm d}x\Big) \\
&=& \log N-2H\bigg(\Big( {\rm Li}_{2}(-\tfrac{Ex}{N}) + (\log x )\log(\tfrac{Ex}{N}+1)\Big){\Big |}_{B}^{B+1} \\
&& -\Big( {\rm Li}_{2}(-\tfrac{Ax}{N}) + (\log x )\log(\tfrac{Ax}{N}+1)\Big){\Big |}_{B}^{D}  - \Big( {\rm Li}_{2}(-\tfrac{Cx}{N}) + (\log x )\log(\tfrac{Cx}{N}+1))\Big) {\Big |}_{D}^{B+1}\bigg),
\end{eqnarray*}
where $H^{-1}= 2\log A+\log(A+1)+\log(B+1)-\log\big(N-(A+1)d\big)-\log\big(N-(d+i)B\big)$ is the normalising constant; cf.~(\ref{eq:NormalizingConstant}) in Theorem~\ref{thm:NaturalExtension}.
\end{proof}

\begin{Example}\label{ex:entropy}{\rm In case $N=2$ our method yields only one plateau with equal entropy which follows from our method. This is the interval $[A,B]=[\frac{\sqrt{33}-5}{2},\sqrt{2}-1] = [0.3722813\cdots , 0.4142136\cdots ]$, which was already found in~\cite{KL}, where it was also determined that for $\alpha\in [A,B]$ we have that $h(T_{\alpha})= 1.137779584292255\cdots$ and $H=3.965116120651161\cdots$.

In case $N=8$ it follows from our method that there are five plateaux of equal entropy; see Table~\ref{table:ExampleEntropy}.
\begin{table}[h!]
\begin{center}
\begin{tabular}{l| c c c c} 

$(d,i)$ & Plateau intervals & Approximation of interval & $H_{\alpha}$ & $h(T_{\alpha})$  \\
\hline
       &                                                                & & \\
(2, 2) & $\left[ \frac{\sqrt{57}-5}{2}, \frac{\sqrt{33}-3}{2}\right]$   & [1.2749,1.3723] & 18.377877038370 & 0.9212748062044  \\
       &                                                                & \\
(4, 6) & $\left[ \frac{3\sqrt{17}-11}{2}, \frac{\sqrt{41}-5}{2}\right]$ & [0.6847,0.7016] & 11.239480662654 & 1.8212263472923 \\
       &                                                                & \\
(5, 11)& $\left[ \frac{3\sqrt{97}-29}{2}, \frac{\sqrt{57}-7}{2}\right]$ & [0.2733,0.2749] & 9.9626774452815 & 2.7933207303296\\
       &                                                                & \\
(6, 22)& $\left[ \frac{\sqrt{321}-17}{2}, \frac{2\sqrt{3}-3}{2}\right]$ & [0.4582,0.4641] & 9.2212359716540 & 2.2547418855378\\
       &                                                                & \\
(7, 57)& $\left[ \frac{3\sqrt{473}-65}{2}, \frac{\sqrt{17}-4}{2}\right]$& [0.1228,0.1231] & 8.7715446381451 & 3.3495778601659\\
       &                                                                & \\
\hline
\end{tabular}
\end{center}
\caption{The pairs of integers $d\geq 1,i\geq 2$, the related plateau intervals $[A,B]$ and constant entropy $h(T_{\alpha})$ for $\alpha\in [A,B]$. Here $N=8$.}\label{table:ExampleEntropy}
\end{table}
In Figure~\ref{fig:SimulationOfEntropyForN=8} a simulation of the entropy for $N=8$ is given as function of $\alpha\in (0,2\sqrt{2}-1]$. The largest of these plateaux of equal entropy from Table~\ref{table:ExampleEntropy} is clearly visible, but the simulation seems to suggest there are other plateaux as well. In general it seems that the entropy is decreasing when $\alpha$ is increasing, but in Figure~\ref{fig:SimulationOfEntropyForN=8}  there is clearly also an interval where the entropy increases. Our method yields one plateau in case $N=2$, but Figures~\ref{FigEntropySimulation} and~\ref{fig:SimulationOfEntropyForN=20} seem to indicate that there are intervals where the entropy is increasing or decreasing. In case $N=20$ our method yields 11 plateaux, the most right-hand one being $[1.844288770,1.898979486]$. However, the simulation in Figure~\ref{fig:SimulationOfEntropyForN=20} seem to indicate there are other plateaux.

\begin{figure}[ht]
\begin{center}
\includegraphics[width=10cm]{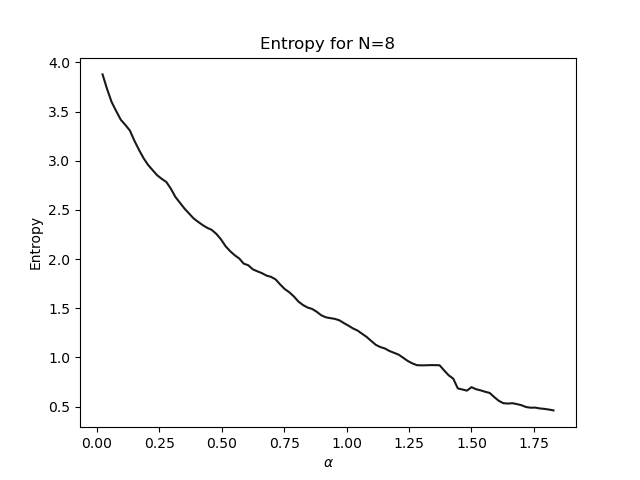}\\
\end{center}
\begin{center}
\caption{A simulation of the entropy of $T_{\alpha}$ when $N=8$.}\label{fig:SimulationOfEntropyForN=8}
\end{center}
\end{figure}

Clearly what we know about $N$-expansions with a finite set of digits is still in its infancy, certainly when compared to the vast body of knowledge about Nakada's $\alpha$-expansions. For these Nakada $\alpha$-expansions Laura Luzzi and Stefano Marmi first showed in~\cite{LM} using simulations that for $\alpha\in [0,g^2]$ there are intervals where the entropy either increases, is constant, or decreases. In~\cite{NN}, Hitoshi Nakada and Rie Natsui showed that there exist decreasing sequences of intervals of $\alpha$, denoted by $(I_n)$, $(J_n)$, $(K_n)$ and $(L_n)$, such that the entropy of $T_{\alpha}$ is increasing on $I_n$, constant on $J_n$ and $L_n$, and decreasing on $K_n$. Furthermore, for $n\in\N$ we have $\frac{1}{n}\in I_n$, and the various intervals are ordered\footnote{Here $I<J$ means that $I\cap J = \emptyset$, and $i<j$ for all $i\in I$ and all $j\in J$.} by $I_{n+1} < J_n < L_n < I_n$. See also the paper by Carlo Carminati and Giulio Tiozzo~\cite{CT}, and in particular Tiozzo's PhD-thesis~\cite{T}.
}\hfill $\triangle$
\end{Example}

\begin{figure}[ht]
\begin{center}
\includegraphics[width=10cm]{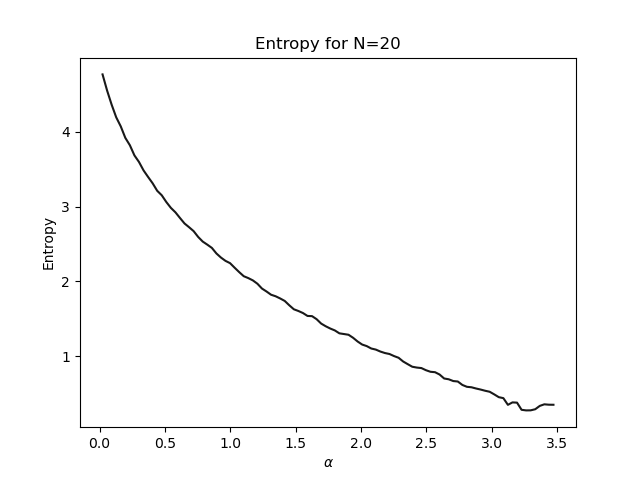}\\
\end{center}
\begin{center}
\caption{A simulation of the entropy of $T_{\alpha}$ when $N=20$.}\label{fig:SimulationOfEntropyForN=20}
\end{center}
\end{figure}

\section*{Acknowledgements}
The research of the first author is supported by the \emph{National Natural Science Foundation of China} (Nos.\ 12071148) and by the \emph{Science and Technology Commission of Shanghai Municipality} (No.\ 22DZ2229014).

\end{document}